\documentclass[11pt]{article}
\usepackage[left=3.5cm, right=3.5cm, top=3.5cm, bottom=3.5cm]{geometry}
\usepackage[T1]{fontenc}
\usepackage[english]{babel}
\usepackage{amsmath}
\usepackage{xcolor}
\usepackage{amssymb}
\usepackage{lmodern}
\usepackage{amsthm}
\nonstopmode
\usepackage{calc}
\usepackage{textcomp}
\usepackage{rotating}
\usepackage{setspace}
\usepackage{amsmath}
\usepackage{here}
\usepackage{latexsym}           
\usepackage{color}
\usepackage{graphicx}
\usepackage{here}
\usepackage{verbatim}
\usepackage[normal]{subfigure}
\usepackage[authoryear]{natbib}

\newtheorem{definition}{Definition}[section]
\newtheorem{lemma}{Lemma}[section]
\newtheorem{theorem}{Theorem}[section]

\newtheorem{assumption}{Assumption}[section]
\theoremstyle{definition}
\newtheorem{example}{Example}[section]
\newtheorem{remark}{Remark}[section]

\title {Properties of stationary cyclical processes}
\author{
   {\L}ukasz Lenart\thanks{
    \textit{{\L}ukasz Lenart was financed from the funds granted to the Krakow University of Economics.}}
\hspace{.2cm}\\
    Department of Mathematics, \\Krakow University of Economics,\\ ul. Rakowicka 27, 31-510 Krakow, Poland, \\email: lukasz.lenart@uek.krakow.pl}
\date{ }

\begin{document}
\maketitle

\linespread{1.5}

\begin{abstract}
\textcolor{black}{The paper investigates the theoretical properties of zero-mean stationary time series with cyclical components, admitting the representation $y_t=\alpha_t \cos \lambda t + \beta_t \sin \lambda t$, with $\lambda \in (0,\pi]$ and $[\alpha_t\,\, \beta_t]$ following some bivariate process. We diagnose that in the extant literature on cyclic time series, a prevalent assumption of Gaussianity for $[\alpha_t\,\, \beta_t]$ imposes inadvertently a severe restriction on the amplitude of the process. Moreover, it is shown that other common distributions may suffer from either similar defects or fail to guarantee the stationarity of $y_t$. To address both of the issues, we propose to introduce a direct stochastic modulation of the amplitude and phase shift in an almost periodic function. We prove that this novel approach may lead, in general, to a stationary (up to any order) time series, and specifically, to a zero-mean stationary time series featuring cyclicity, with a pseudo-cyclical autocovariance function that may even decay at a very slow rate. \textcolor{black}{The proposed process fills an important gap in this type of models and allows for flexible modeling of amplitude and phase shift}.} 
\end{abstract}

\section{Introduction}
\label{sec:Intro}
While the concept of stationary time series is generally well-established in the literature, there remains a gap in the theory regarding zero-mean stationary time series $\{y_t:\,t\in \mathbb{Z}\}$ (we simplify to $\{y_t\}$) that exhibit cyclic features with unknown frequencies. Namely, little attention seems to be paid to the theoretical properties of the underlying {\it{amplitude process}}, whose statistical characteristics (e.g., expectation, if exists) can measure the maximum distance between extreme deviations of the observed process. For instance, in existing Gaussian processes with pseudo-cyclical autocovariance function (see \cite{Hannan1970}, \cite{Proietti2023}), the relationship between the expectation and standard deviation of the {\it{amplitude process}} is linear, which appears very restrictive. This arises from the fact that the {\it{amplitude process}} is not directly defined, and therefore, its statistical properties are consequences of the assumptions made, often intended to ensure stationarity or facilitate statistical inference (e.g., Gaussianity). The main theoretical limitations (in terms of flexibility) for the amplitude process in existing approaches stem from both the assumption of Gaussianity and the manner in which such cyclical models are specified. Below, we delve into details. In particular, Subsection 1.1 discusses existing zero-mean stationary processes with cyclic behaviors of unknown frequencies, highlighting their theoretical properties, limitations and potential extensions. In Subsection 1.2 we focus on the ambiguity of the amplitude process specification, a problem apparently absent from the existing literature. The following subsection reviews other, non-stationary approaches to modeling cyclical data. Subsection 1.4 presents what is the main contribution of this work, namely a novel specification of the cyclic process, which enhances the flexibility of the amplitude process. Finally, Subsection 1.5 outlines the structure of our paper.  

Before we move one, let us clarify the sense in which the term \textit{stationary} is used throughout the paper. In general, a time series $\{y_t:\, t \in \mathbb{Z}\}$ is said to be \textit{stationary up to order} $m$ if for any $s \leq m$, and $\tau_1,\tau_2,\ldots,\tau_s \in \mathbb{Z}$, and $t \in \mathbb{Z}$, the moment $E(y_{t+\tau_1}y_{t+\tau_2}\ldots y_{t+\tau_s})$ exists and does not depend on time $t$:
\begin{equation}
E(y_{t+\tau_1}y_{t+\tau_2}\ldots y_{t+\tau_s})= E(y_{\tau_1}y_{\tau_2}\ldots y_{\tau_s})<\infty.   
\end{equation}  
In particular, by referring to a process as stationary, we mean a time series that is stationary up to order two (or "weakly stationary", or "stationary in a wide sense").

\subsection{Processes review} 
\label{sec:1.1}

Modeling cyclic phenomena has a long history. The pioneering work of \cite{Yule1927} provides important theoretical preliminaries for modeling cyclic data by introducing disturbances to the right-hand side of the trigonometric identity $A \sin \theta t = 2A \cos \theta \sin \theta (t-1) - A \sin \theta (t-2),$ where $t \in \mathbb{Z}$, $\theta \in [0,2\pi)$, and $A \in \mathbb{R}$. Ultimately, the approach yields a non-stationary $\text{AR}(2)$ process $\{u_t\}$ of the form $u_t=2 \cos \theta  u_{t-1} - u_{t-2} + \epsilon_t$, with a white noise $\{\epsilon_t\}$ and complex roots $e^{i \theta}$ and $e^{-i \theta}$.

Although the above-mentioned process is non-stationary, it has prompted research for a stationary version; see \cite{Kendall1945} for some early-stage developments and review of then existing other approaches. In \cite{Hannan1964}, a stationary seasonal cycle process was defined as $y_t=\alpha_t \cos \tilde \lambda t + \beta_t \sin \tilde \lambda t$, where $\tilde \lambda$ is a seasonal frequency and ${\alpha_t}$ and ${\beta_t}$ are uncorrelated first-order zero-mean Gaussian autoregressive processes with the same autoregressive coefficient $\rho$ (where $|\rho|<1$) and the same variance of the underlying white noises. However, the frequency $\tilde \lambda$ does not necessarily have to be seasonal to ensure the stationarity of such a process, which was quickly adopted in future research. 

In the decades to follow, new concepts of stationary cyclic processes were theoretically investigated. \cite{Hannan1970} proposes a process that is the sum $\sum_{j=1}^{K}z_{t,j}$ of $K$ cyclical components:
\begin{equation}
z_{t,j}=\alpha_{t,j} \cos \lambda_j t + \beta_{t,j} \sin \lambda_j t,
\label{Hannan1964B}
\end{equation}
where for all $j=1,2,\ldots,K$, the frequency $\lambda_j \in (0,\pi]$ is seasonal, and $[\alpha_{t,j}\,\,\, \beta_{t,j}]$ is a bivariate stationary Gaussian process. However, similarly to the single-frequency case, each $\lambda_j$ can be actually any frequency within the interval $(0,\pi]$, without losing the stationarity of the process.

Next, in \cite{Harvey1985}, \cite{Harvey1989}, and \cite{HarveyJaeger1993}, a \textit{stochastic cycle} model with a Gaussian $\text{ARMA}(2,1)$ representation was developed. \cite{HarveyTrimbur2003} and \cite{Trimbur2006} generalized the concept to an \textit{nth-order stochastic cycle} model with a Gaussian $\text{ARMA}(2n,2n-1)$ representation. Yet some other alternative specifications with $\text{ARMA}$ representations based on Gaussian white noise were considered in \cite{LuatiProietti2010}. Further, \cite{GARMA1998} and \cite{Smallwood2003} consider a $k$-factor Gegenbauer ARMA (GARMA) model. Finally, \cite{Proietti2023} extended the approach presented in \ref{Hannan1964B} with a non-seasonal frequency and with $\alpha_t$ and $\beta_t$ being independent fractionally differenced white noise processes (see \cite{Andel1986}) with Gaussian white noises with the same variance. A brief characterization of each of the above-mentioned models is given in Section 3. 

All of the above concepts share a common feature: they are (or can be, as we demonstrate it later in this work) represented as (see \cite{Hannan1970}, \cite{Proietti2023}):
\begin{equation}
y_t=\alpha_t \cos \lambda t + \beta_t \sin \lambda t,
\label{Hannan1964}
\end{equation}
where $[\alpha_t \,\,\,\ \beta_t]$ forms a zero-mean bivariate stationary Gaussian process. This representation naturally allows for interpretation of the amplitude and phase shift in terms of the bivariate Gaussian process $[\alpha_t\,\,\, \beta_t]$. Based on (\ref{Hannan1964}), the amplitude process can be defined as:
\begin{equation}
\text{AMP}_t=\sqrt{\alpha_t^2+\beta_t^{2}},
\label{amplituda}
\end{equation}
which means that the basic characteristics (at a fixed time $t \in \mathbb{Z}$), such as expectation and standard deviation, result from the properties of the distribution of $[\alpha_t \,\,\, \beta_t]$. From a practical standpoint, for a zero-mean stationary cyclic process to be sufficiently flexible in empirical modeling, basic characteristics such as the expectation and standard deviation of the amplitude fluctuations (as in Equation \ref{amplituda}) should not be closely related to each other. In other words, the model specification should enable these quantities to behave freely, with no restraints stemming from their interrelations. 

To highlight the problem, note that the assumption $[\alpha_t \,\,\, \beta_t] \sim N_{2}(\boldsymbol{0},\sigma^2 \textbf{I}_2)$ in representation (\ref{Hannan1964}), ensuring the stationarity of $\{y_t\}$ and ${\text{AMP}_t}$, gives $\mu_{\text{AMP}}=E(\text{AMP}_t)= \sigma_{\text{AMP}}\sqrt{\frac{\pi}{4-\pi}}=\sqrt{\frac{\pi}{2}}\sigma$, where (we omit the straightforward calculations) $\sigma_{\text{AMP}}=\sqrt{E(\text{AMP}_t - \mu_{AMP})^2}=\sqrt{\frac{4-\pi}{2}}\sigma$ is the amplitude's standard deviations. Both quantities, characterizing quite distinct aspects (location and dispersion) of the amplitude process, are driven by the common parameter of $\sigma$, which is very restrictive ($\frac{\mu_{\text{AMP}}}{\sigma_{\text{AMP}}}=\sqrt{\frac{\pi}{4-\pi}}$),  and in empirical modeling may affect statistical inference (e.g. push towards non-stationarity). For that matter, we argue that it can be seen in the empirical results obtained recently by \cite{maddanu2022modelling} (see Figure 6 therein, presenting the amplitude process), \cite{Proietti2023} (see Figure 3 therein, presenting the amplitude process), and by \cite{koopman2008measuring}, \cite{harvey2007trends} (see Figure 9 therein), \cite{GARMA1998}, for example. Conceivably, it may be due to these strong limitations that these models have not gained much popularity in applications, apart from a few studies in the area of business cycles analysis, where fluctuations are generally very irregular, with time-varying, irregular amplitude and phase shift. Moreover, it's worth noticing already that some natural generalizations, such as the bivariate Student's \textit{t}-distribution, skew normal distributions, and other typical distributions, present no valid a solution to the problem, as they exhibit similar strong relations between the expectation and standard deviation of the amplitude process or leads to loss of stationarity, as will be demonstrated in Section 3. Then, we will propose a way to overcome the issue by specifying a novel way for construction of zero-mean stationary cyclic process.

As already argued above, examination of the moments' structure of the amplitude process (and the phase shift, as well) has missed the researchers' scope of attention so far. However, interestingly enough, pertinent derivations have already preoccupied researchers of an unrelated area of electrical and electronics engineering, where, in general, for a bivariate random variable $[X \,\,\, Y]$, the stochastic properties of the amplitude, defined as $\sqrt{X^2+Y^2}$, and phase shift, $\arctan(X/Y)+\mathbb{I}\{X<0\}\pi$, are of key interest; see, e.g., \cite{Nadarajah2016}, \cite{Balakrishnan2009}, and references therein. Some of distributions considered in \cite{Balakrishnan2009} are also of our interest later in this work.

To end this subsection, let us note that similar arguments and considerations as the ones raised above with respect to the amplitude would also apply to the phase shift. Nonetheless, for the sake of our current work's volume, we defer these for future research, and focus here only on some theoretical developments for the amplitude process.

\subsection{The amplitude process ambiguity} Another issue pertaining to zero-mean stationary processes, one that has not been raised in the existing literature (to the best of our knowledge), is the ambiguity of the amplitude, arising from the fact that for any fixed frequency $\lambda\in(0,\pi]$, any zero-mean stationary process $\{y_t\}$ can be represented in infinitely many ways as 
\begin{equation}
y_t=\alpha_t \cos \lambda t + \beta_t \sin \lambda t,     
\end{equation}
with 
\begin{equation}
\begin{split}
[\alpha_t \,\, \beta_t]'=\textbf{R}(-t \lambda) [y_t \,\, y_t^*]'=\left[\begin{array}{cc}
 y_t \cos \lambda t - y_{t}^* \sin \lambda t   \\
 y_t \sin \lambda t + y_{t}^* \cos \lambda t
\end{array}\right]',
\end{split}    
\label{alpha_rep0}
\end{equation} where $\{y_t^*\}$ is any real valued process and $\textbf{R}(z)$ is a rotation matrix of the form 
$\textbf{R}(z)=\left[
\begin{array}{cc}
 \cos z  & \sin z  \\
- \sin z & \cos z \\ 
\end{array}
\right]$, with $z \in \mathbb{R}$. Based on the above representation, the amplitude process related to the frequency $\lambda$ for any stationary process $\{y_t\}$ is $ \text{AMP}_t=\sqrt{\alpha_t^2+\beta_t^{2}}=\sqrt{y_t^2+y_t^{*2}}$. \textcolor{black}{To illustrate the ambiguity problem, let us consider a simple example below.}
\begin{example}
\textcolor{black}{Let $\{y_t\}$ follow a simple $\text{AR}(1)$ process: $y_t=\rho y_{t-1} + \epsilon_{t}$, where $|\rho|<1$, and $\epsilon_t$ is a zero-mean Gaussian white noise with variance $\sigma^2>0$. Three cases of the process $y_t^*$ can be considered:
\begin{itemize}
    \item[C1] For $y_t^*= a y_t$, where $a \in \mathbb{R}$, the amplitude process is a stationary (up to any order) process of the form $\text{AMP}_t=\sqrt{1+a^2}|y_t|$ (with the unconditional rescaled chi distribution with one degree of freedom). Here, $[\alpha_t \,\, \beta_t]$ is a zero-mean process but not stationary, as the second-order (and higher-order) moment structure is time-dependent (calculations are omitted);  
    \item[C2] Assume now that $y_t^*= \rho y_{t-1} + \epsilon_{t}^*$, where $\epsilon_{t}$ is a zero-mean Gaussian white noise with variance $\sigma^2$ and independent from $\epsilon_{t}$. Then, the amplitude process is a stationary (up to any order) process with the unconditional rescaled chi distribution with two degrees of freedom. Elementary calculations (omitted for brevity) show that $[\alpha_t \,\, \beta_t]$ is also stationary up to any order.
    \item[C3] Finally, assume that $y_t^*= y_{t-1}^* + \epsilon_t$. Then, the amplitude process $\text{AMP}_t=\sqrt{y_t^2 + y_t^{*2}}$ is not stationary, and neither is $[\alpha_t \,\, \beta_t]$.
\end{itemize}
}
\end{example}
\textcolor{black}{To limit possible specifications of the amplitude, in this paper, we will require both the amplitude process $\text{AMP}_t$ and $[\alpha_t \,\,\, \beta_t]$ to be stationary. In particular,} we derive the conditions for ${y_t^*}$ under which $[\alpha_t \,\,\, \beta_t]$ is a zero-mean stationary process. 

\subsection{Other non-stationary cyclic processes}
It should be duly noted that there exist other approaches to modeling cyclical fluctuations, which reach beyond the class of stationary time series by either introducing some amplitude or phase shift modulations or employing the idea of almost periodicity (also referred to as cyclostationarity). Regarding the latter, a function $f(t): \mathbb{Z} \to \mathbb{R}$ is called almost periodic if for any $\epsilon >0$, there exists a positive integer $L_{\epsilon}$ such that among any $L_{\epsilon}$ consecutive integers, there is an integer $p_{\epsilon}$ for which $\sup\limits_{t \in \mathbb{Z}}|f(t+p_{\epsilon})-f(t)|<\epsilon$ (see \cite{Corduneanu1989}). Cyclic features can be modeled by means of almost periodically correlated (APC) time series (nonstationary, in general), the class of which encompasses, among others, periodically correlated (PC) series as well covariance stationary sequences (the latter being the object of interest in this work). We say that a time series $\{y_t:\,\, t \in \mathbb{Z}\}$ with finite second moments is PC with a period length $T>1$ if the mean function $\mu(t)=E(y_t)$ and the autocovariance function $B(t,\tau)=\text{cov}(y_t,y_{t+\tau})$ are both periodic at $t \in \mathbb{Z}$ with the same period length $T$ (for any $\tau \in \mathbb{Z}$). Meanwhile, for an APC time series, it is assumed that the mean and autocovariance functions are almost periodic functions of $t$.
The above approach assumes almost periodicity in the first and second moments, which is beyond the scope of our interest in this work. Although the APC and PC processes have been quite popular in the relevant literature, with the examples including \cite{Antoni2009}, \cite{GardnerNapolitanoPaura2006}, \cite{Napolitano2012}, our current attention focuses solely on \textit{stationary} time series of a cyclical nature, i.e., with a zero mean and a {\it pseudo-cyclical} autocovariance function.

Another common approach to modeling cyclicity is by means of amplitude-modulated (AM) time-warped (TW) APC processes (see \cite{NapolitanoGardner2016}, \cite{Napolitano2017}, \cite{Gardner2018}, \cite{Napolitano2019}, \cite{Napolitano2022}) by considering $y(t)=a(t)x(\psi(t))$, where $a(t)$ is a deterministic time-varying amplitude, $x(t)$ is a continuous-time PC process with a real-valued period length $T_0>0$, and $\psi(t)=t+\epsilon(t)$, with $\epsilon(t)$ being some slowly varying and differentiable function. In \cite{Gardner2018}, it was assumed that $a(t) \equiv 1$. The generalization to any deterministic function $a(t)$ was developed in \cite{Napolitano2022} in electrocardiogram modeling. Some related approaches were considered in \cite{Das2021}. Another approach with a time-variable (irregular) rhythm was developed recently in \cite{Lupenko2023}. The common feature of these approaches is the use of purely deterministic functions to modulate the amplitude or period length (time-warped function). The main problem with these generalizations, however, is the lack of theoretical results relating to the possibility of statistical inference. Arguably, they also seem rather sophisticated and thus has found only few applications so far.

\subsection{Cyclic process proposition} In this paper, to address the flexibility and ambiguity of the amplitude process, we have decided to model directly the magnitude and phase shift of cyclical fluctuations. Specifically, the flexibility of our specification draws from an explicit modeling the mean and standard deviation of the amplitude process. We achieve this goal directly by introducing zero-mean stochastic processes $\{A_t:\, t \in \mathbb{Z}\}$ (to disrupt the amplitude $a$) and $\{P_t:\, t \in \mathbb{Z}\}$ (to disrupt the phase shift $\psi$) in the almost periodic function $a\sin(\lambda (t + \psi))$, for $t \in \mathbb{Z}$,  with an unknown frequency $\lambda \in (0,\pi]$. Specifically, our proposition takes the form: \begin{equation}
y_{t}=(a+A_t)\sin[\lambda(t+\psi + P_t)],
\label{model_prep}
\end{equation} and can be proven a zero-mean, stationary (up to any order $m$) time series with a pseudo-cyclical autocovariance function, under suitable conditions derived in this article. We devote Section 4 to discuss the properties of our process in detail, and demonstrate that it is fit for for adequate modeling of the amplitude and phase shift.  

\textcolor{black}{Notice that our model structure (\ref{model_prep}) does not represent something actually distinct from representation (\ref{Hannan1964}), as elementary calculations show that
\begin{equation}
\begin{split}
y_t & = (a+A_t) \sin(\lambda(\psi + P_t)) \cos \lambda t+ (a+A_t) \cos(\lambda(\psi + P_t)) \sin \lambda t   \\
& = \tilde \alpha_t \cos \lambda t+ \tilde \beta_t \sin \lambda t, 
\end{split}
\end{equation}
with $[\tilde \alpha_t\,\, \tilde \beta_t]=[ R_t 
\sin \theta_t\,\,\, R_t \cos \theta_t ]=R_t [ 
\sin \theta_t\,\,\, \cos \theta_t ]$, $\theta_t=\lambda(\psi + P_t)$ and $R_t=a+A_t$. If $R_t$ and $\theta_t$ are assumed to be mutually independent, and that $\theta_t$ follows a uniform distribution on the interval $(0,2 \pi)$, then through some elementary calculations it can be shown that the characteristic function of $[\tilde \alpha_t\,\, \tilde \beta_t]$ at point $\textbf{z}=[z_1\,\, z_2]\in \mathbb{R}^2$ depends only on $\| \textbf{z}\| = \sqrt{z_1^2+z_2^2}$. 
However, here we do not assume in our proposed model that the distribution of $[\tilde \alpha_t\,\, \tilde \beta_t]$ is Gaussian, for it would impose a very strong restriction. As a solution, we focus solely on the distribution of $[\tilde \alpha_t\,\, \tilde \beta_t]$ that results from the ones assumed for $A_t$ and $P_t$.}

\subsection{Article structure}
The structure of the article is as follows. Section 2 presents the theoretical background for zero-mean, stationary processes with a pseudo-cyclical autocovariance function, filling an important gap in the literature. Next, in Section 3, based on the results from Section 2, we characterize existing stationary processes with a pseudo-cyclical autocovariance function, identify possible problems and limitations, and diagnose their causes. In Section 4, we propose a novel concept of a stochastic cycle, along with a comprehensive exposition of its theoretical properties, with a particular emphasis on stationarity (up to order $m$), autocovariance function and power spectral density function. Finally, note that the proofs of all the theorems formulated in the paper, are deferred to the Appendix. 

\section{Theoretical results}

\subsection{Single-frequency case}
We begin with the process 
\begin{equation}
y_t = \alpha_t \cos \lambda t + \beta_t \sin \lambda t,    
\label{proc_yt}
\end{equation}    
featuring a single frequency $\lambda \in (0,\pi]$, under general assumptions regarding the bivariate process $[\alpha_t\,\,\,\beta_t]$, not necessarily ensuring stationarity of $\{y_t:\, t \in \mathbb{Z}\}$.

\begin{theorem}
Let  $\lambda \in (0,\pi]$ and $\{y_t:\, t \in \mathbb{Z}\}$ be defined as
\begin{equation}
 y_t = [\alpha_t\,\,\beta_t]  \cdot [\cos \lambda t \,\, \sin \lambda t]'=\alpha_t   \cos \lambda t + \beta_t \sin \lambda t,
 \label{eq21}
\end{equation}
where $[\alpha_t\,\,\beta_t]$ is a bivariate stationary time series with $1\times 2$ mean vector $\boldsymbol{\mu}=[\mu_1\,\,\mu_2]$ and $2\times 2$ covariance matrix $\boldsymbol{\Omega}(\tau)=E\bigg(\big([\alpha_{t+\tau}\,\,\beta_{t+\tau}]-\boldsymbol{\mu}\big)'\big([\alpha_t\,\,\beta_t]-\boldsymbol{\mu}\big)\bigg)=[\omega_{ij}(\tau)]_{2\times 2}$, $\tau \in \mathbb{Z}$. Then, $\{y_t:\, t \in \mathbb{Z}\}$ given by (\ref{eq21}) is: 
\begin{itemize}
    \item[i)] an APC process with an almost periodic mean function 
\begin{equation}
\mu_{y}(t)=E(y_t)=  [\mu_1\,\,\mu_2]  \cdot [\cos \lambda t \,\, \sin \lambda t]'=\mu_1   \cos \lambda t + \mu_2 \sin \lambda t,  
\label{eq:mean}
\end{equation}
and an almost periodic covariance function of the form
\begin{equation}
\begin{split}
E\big((y_t-\mu_{y}(t)) & (y_{t+\tau}-\mu_{y}(t+\tau))\big)\\&= [\cos \lambda t \,\, \sin \lambda t] \boldsymbol{\Omega}(\tau) [\cos \lambda (t+\tau) \,\, \sin \lambda (t+\tau)]';
\end{split}
\end{equation}
\item[ii)] a zero-mean stationary process if and only if for any $\tau \in \mathbb{Z}$
\begin{equation}
\boldsymbol{\mu}=\boldsymbol{0} \,\,\,\,\, \text{ and } \,\,\,\,\,\, \omega_{11}(\tau)=\omega_{22}(\tau)   \,\,\,\,\,\, \text{ and } \,\,\,\,\,\,  \omega_{12}(\tau)=-\omega_{21}(\tau). 
\label{eq:eq}
\end{equation}
\end{itemize}
\label{tw21}
\end{theorem}

Note that the mean function (\ref{eq:mean}) is constant if and only if $\mu_1=\mu_2=0$. Moreover, $\{y_t\}$ is stationary only if $\omega_{11}(\tau)=\omega_{22}(\tau)$ and $\omega_{12}(\tau)=-\omega_{21}(\tau)$. Then, we easily obtain $E(y_t y_{t+\tau})=\omega_{11}(\tau) \cos \lambda \tau  + \omega_{12}(\tau) \sin \lambda \tau$, which is a more general result than the one commonly found in the literature (see \cite{HarveyTrimbur2003}, \cite{Trimbur2006}, \cite{Proietti2023}), where it is additionally (and unintentionally) assumed that $\omega_{12}(\tau)=0$, which is only one from a sufficient but not necessary condition for the stationarity of $\{y_t\}$. \textcolor{black}{However, this is not the only sufficient condition, as the remark below shows.}

\begin{remark}
\textcolor{black}{Note that independence is not the only sufficient condition for the stationarity up to order $m>2$, and the shape of the distribution for $[\alpha_t\,\, \beta_t]$ plays a crucial role here. Let us consider a few simple examples of independent random variables $\alpha_t$ and $\beta_t$ below, for which the independence guarantees only the stationarity up to order 2 and not necessarily higher.
\begin{itemize}
    \item {\it{Skewed distribution}} with zero expected value, variance equal to one and skewness equal to $\zeta \not = 0$. Elementary calculations (omitted here) yield $E(y_t^3)= \zeta \left(\sin ^3(\lambda  t)+\cos ^3(\lambda  t)\right)$,  which depends on $t$, indicating that $\{y_t\}$ is not a stationary time series up to order 3.
    \item {\it{Logistic distribution}} with zero mean and variance $\frac{\pi ^2 \nu ^2}{3}$ features (we omit the calculations) $E(y_t^4)=\frac{1}{30} \pi ^4 \nu ^4 (\cos (4 \lambda  t)+13)$, which depends on time $t$, implying that such a process is not stationary up to order 4.
    \item {\it{Irwin–Hall distribution} } (which is the distribution of a sum of $n$ independent random variables that are uniformly distributed on the same interval $(-a,a)$; see \cite{johnson1995continuous}) features $E(y_t^4)=\frac{1}{30} a^4 n (10 n-\cos (4 \lambda  t)-3)$, which depends on $t$, indicating that such a process is not stationary up to order 4.
    \item {\it{Scale mixture of the normal distributions}} (see \cite{andrews1974scale}). Assume that $\alpha_t$ and $\beta_t$ are independent and follow the same unconditional distribution, which belongs to the scale mixture of normal distributions with a stochastic representation of the form $R\cdot N$, where $N$ and $R$ are independent random variables, with $N$ following the standard Gaussian distribution and $R$ following a continuous distribution on the interval $(a,b)$, where $0 \leq a < b \leq \infty$. This class includes such distributions as the Student's t distribution, the slash distribution, and many others. Assuming that $E(R^4)<\infty$, elementary calculations yield: $E(y_t^4)=1+E\big[(R_{1}^2-R_{2}^2)^2 \cos^4(\lambda t )\big] $, where $R_1$ and $R_2$ are independent random variables with the same distribution as $R$. This formula demonstrates that $E(y_t^4)$ depends on $t$, indicating that no distribution belonging to this class ensures stationarity up to order 4 (except for the limiting case of $a \to b > 0$, under which $\alpha_t$ and $\beta_t$ tend toward the Gaussian case).  
 \end{itemize}}
 \end{remark}

The fundamental issue we face with the above two-dimensional distributions for $[\alpha_t\,\,\beta_t]$ is that, upon rotating the vector $[\alpha_t\,\,\beta_t]$ by angle $\lambda t$ with the rotation matrix $\textbf{R}(\lambda t)$, the resulting distribution of $[\tilde \alpha_t\,\, \tilde \beta_t]=\textbf{R}(\lambda t)[\alpha_t\,\,\beta_t]$ changes in $t$ and thus is time-dependent. We find that (and elaborate on that in what follows) a solution to this problem lies in assuming such a probability distribution for $[\alpha_t\,\,\beta_t]$ that remains invariant under rotation by any angle. Examples of such distributions include those with their probability density function at $(x,y)\in \mathbb{R}^2$ for $[\alpha_t\,\,\beta_t]$ being proportional to $f(x^2+y^2)$, where $f:[0,\infty) \to [0,\infty)$ is a real-valued function. This leads to a well-known class of the spherical distributions (see \cite{alma991031536379705251}, \cite{gupta2013elliptically}, \cite{fang2018symmetric}).

The following theorem shows that the only distribution for $[\alpha_t \,\, \beta_t]$ such that $\alpha_t$ and $\beta_t$ are independent, and ${y_t}$ is stationary up to any order, is the Gaussian distribution.

\begin{theorem}
Take any $\lambda \in (0,\pi]$. If $\alpha_t$ and $\beta_t$ are mutually independent for any $t \in \mathbb{Z}$, and the process $y_t = \alpha_t \cos \lambda t + \beta_t \sin \lambda t$ is stationary up to any order, then $\alpha_t$ and $\beta_t$ have the same Gaussian distribution. 
\label{tw_Gaussian_independent}
\end{theorem}

In the following theorem, we formulate the condition for the distribution of the bivariate process $[\alpha_t \,\, \beta_t]$ which ensure the strict stationarity for $\{y_t\}$.

\begin{theorem}
Assume that for any positive integer $m$ and any $\boldsymbol{\tau}=(\tau_1,\tau_2,\ldots,\tau_s) \in \mathbb{Z}^s$ such that $s \leq m$ and $\tau_1 < \tau_2 < \ldots < \tau_s$, the probability distribution function of the vector $\textbf{S}_{t,\boldsymbol{\tau}}=[\alpha_{t+\tau_1}\,\, \beta_{t+\tau_1} \,\, \alpha_{t+\tau_2}\,\, \beta_{t+\tau_2} \ldots \alpha_{t+\tau_s}\,\, \beta_{t+\tau_s} ]$ at point $(x_{1},x_{1}^{*},x_{2},x_2^{*},\ldots,x_s,x_s^{*}) \in \mathbb{R}^{2s}$ does not depend on $t$ and has the form $f_{\boldsymbol{\tau}}(x_1^2+x_1^{*2},x_2^2+x_2^{*2},\ldots,x_s^2+x_s^{*2})$, where $f_{\boldsymbol{\tau}}:(x_1,x_{1}^{*},x_2,x_2^{*},\ldots,x_s,x_s^{*}) \in \mathbb{R}^{2s} \to \mathbb{R}$. Then the time series ${y_t}$ is strictly stationary.
\label{strictly_stationary}
\end{theorem}

The following theorem addresses the properties of the $\{y_t\}$ process when $\alpha_t$ and $\beta_t$ are dependent and given by linear filters based on coordinates of spherically distributed bivariate IID sequence.


\begin{theorem}
Let $\alpha_t=\sum\limits_{k=0}^{\infty} \psi_{k} \epsilon_{t}$ and $\beta_t=\sum\limits_{k=0}^{\infty} \psi_{k} \epsilon_{t}^*$, with $\sum\limits_{k=0}^{\infty}|\psi_k|<\infty$, where $[\epsilon_{t}\,\,\epsilon_{t}^*]$ is IID with a zero-mean spherical distribution with the probability distribution function $f(x^2+y^2)$ at $(x,y)\in \mathbb{R}^2$, where $f:[0,\infty) \to [0 ,\infty)$ is some real-valued function. Then,
\begin{itemize}
    \item[i)] $[\alpha_t\,\,\beta_t]$ follows a spherical distribution;
    \item[ii)] the process $y_{t}=\alpha_t \cos \lambda t + \beta_t \sin \lambda t$ can be represented as $y_t=\sum\limits_{k=0}^{\infty} \psi_{k} [ \cos (\lambda k) \zeta_{t-k} +\sin (\lambda k) \zeta_{t-k}^*] $, where $[\zeta_t\,\,\zeta_t^*]'=\textbf{\emph{R}}(\lambda t) [\epsilon_{t}\,\,\epsilon_{t}^*]'$ is a zero-mean white noise with the same probability distribution function as $[\epsilon_{t}\,\,\epsilon_{t}^*]'$;
    \item[iii)] for any positive integer $m$, if $E|\epsilon_t|^m<\infty$, then $y_t$ is stationary up to order $m$;
\item[iv)] $y_t$ is strictly stationary.
\end{itemize} 
\label{tw_representation_spherical}
\end{theorem}
Notice, however, that should one opt for the components to be independent, then there exists only one such  spherical distribution: the multivariate normal distribution (with zero mean and the same variance for both coordinates) which is known as Maxwell's theorem.

\subsection{Amplitude process properties and inverse coefficient of variation}

As already mentioned in Subsection 1.2, if the process $\{y_t\}$ is stationary, then the representation of the form (\ref{eq21}) is not unique even if we require that $[\alpha_t \,\, \beta_t]$ is a zero-mean stationary process. We take this into account in the following theorem.

\begin{theorem}
Let $\{y_t\}$ be any zero-mean stationary time series. Then, for any frequency $\lambda \in (0,\pi]$, the process $\{y_t\}$ can be represented as  
\begin{equation}
 y_t = \alpha_t \cos \lambda t + \beta_t \sin \lambda t,  
 \label{y_rep}
\end{equation}
where $[\alpha_t \,\, \beta_t]$ is a zero-mean stationary process of the form 
\begin{equation}
\begin{split}
[\alpha_t \,\, \beta_t]'=\textbf{\emph{R}}(-t \lambda) [y_t \,\, y_t^*]'
\end{split}    
\label{alpha_rep}
\end{equation}
and $[y_t\,\, y_t^*]$ is a bivariate zero-mean stationary process with the \textcolor{black}{covariance matrix} $\boldsymbol{\Gamma}(\tau)=E\bigg( [y_{t+\tau}\,\, y_{t+\tau}^*]'[y_t\,\, y_t^*] \bigg)=[\gamma_{ij}(\tau)]_{2 \times 2}$ such that for any $\tau \in \mathbb{Z}$ we have $\gamma_{11}(\tau)=\gamma_{22}(\tau)$ and $\gamma_{12}(\tau)=-\gamma_{21}(\tau)$. In addition,  
\begin{equation}
\boldsymbol{\Omega}(\tau)=E([\alpha_{t+\tau} \, \beta_{t+\tau}]'[\alpha_t \,\, \beta_t])=  \boldsymbol{\Gamma}(\tau) \textbf{\emph{R}} (-\lambda \tau). 
\end{equation}
In particular, if for $\tau_0 \in \mathbb{Z}$ we have that $ \gamma_{11}(\tau_0)   \sin \lambda \tau =  \gamma_{12}(\tau_0) \cos \lambda \tau,$ 
then $\boldsymbol{\Omega}(\tau_0)$ is a diagonal matrix of the form  
\begin{equation}
 \boldsymbol{\Omega}(\tau_0)=(\gamma_{11}(\tau_0)   \cos \lambda \tau_0 +  \gamma_{12}(\tau_0) \sin \lambda \tau_0)\textbf{\emph{I}}_2.   
 \label{eq:omega2}
\end{equation}
\label{tw:22}
\end{theorem}

From the theorem above it follows that the amplitude process related to the frequency $\lambda$ takes the form 
\begin{equation}
\text{AMP}_t=\sqrt{\alpha_t^2 + \beta_t^2}=\sqrt{y_t^2+y_t^{*2}},    
\end{equation}
which means that the distribution of amplitude process at time $t$ is determined only by the distribution of $y_t$ and $y_t^*$ or, equivalently, the distribution of $[\alpha_t \,\, \beta_t]$.

Apart from the conditions for $\{[\alpha_t\,\,\,\beta_t]: \, t \in \mathbb{Z}\}$ ensuring the stationarity of $\{y_t\}$, it seems natural to consider that the amplitude process $\{\sqrt{\alpha_t^2+\beta_t^2}: \, t \in \mathbb{Z}\}$ is a stationary process, as well. Therefore, based on the theorem above, we formulate the following assumption.
\begin{assumption}
Let $[\alpha_t \,\,\, \beta_t]$ be a zero-mean stationary bivariate process such that $\{y_t\}$ given by (\ref{proc_yt}) is stationary and the amplitude process $\emph{AMP}_t=\sqrt{\alpha_t^2 + \beta_t^2}$ is stationary with continuous probability distribution at time $t \in \mathbb{Z}$ given by $f_{t}(x,x^*)$ at point $(x,x^*) \in \mathbb{R}^2$.
\label{ass1}
\end{assumption}

Under Assumption \ref{ass1} let $f_t(x,x^*)$ be the probability distribution function (pdf) of $[\alpha_t \,\, \beta_t]$ at time $t$. Then, the pdf of the amplitude $\text{AMP}_t$ is simply
\begin{equation}
f_{\text{AMP}_t}(\xi)=\int_{0}^{2 \pi} \xi  f_t(\xi \cos \theta, \xi \sin \theta) \text{d}\theta,
\label{pdfAMP}
\end{equation} where $\xi \in [0,\infty)$. If we additionally assume that the distribution of $[\alpha_t \,\, \beta_t]$ at time $t$ is spherical such that $f_{t}(x,x^*)=g_t(x^2+x^{*2})$ for some $g_t:\mathbb{R}\to \mathbb{R}$, then we have
\begin{equation}
f_{\text{AMP}_t}(\xi)=2 \pi \xi  g_t(\xi^2).
\label{pdfAMP}
\end{equation}
Under Assumption \ref{ass1}, we introduce the notation: $\mu_{\text{AMP}}=E(\text{AMP}_t)$ and  $\sigma_{\text{AMP}}^2=Var(\text{AMP}_t)$. Then, the inverse of a common coefficient of variation of the amplitude does not depend on $t$ and can be written as    
\begin{equation}
\text{ICV}_{\text{AMP}}=\frac{\mu_{\text{AMP}}}{\sigma_{\text{AMP}}}.
\label{relationAMP}
\end{equation}
If $\text{ICV}_{\text{AMP}}$ is a constant not depending on any parameters, then (\ref{relationAMP}) yields a strictly linear relation between the amplitude's mean and standard deviation. Notably, this is exactly the case under the typical choice of a bivariate zero-mean, uncorrelated and common-variance Gaussian distribution for $[\alpha_t \,\, \beta_t]$. Then, the amplitude follows simply a chi distribution with two degrees of freedom, and thus 
\begin{equation}
\text{ICV}_{\text{AMP}}=\sqrt{\frac{\pi}{4-\pi}},
\label{ICV_Gaussian}
\end{equation} which introduces an extremely tight, deterministic relation between the amplitude's first two moments. 

\begin{remark}
In relation to the Gaussian case mentioned above, let us assume that second order zero-mean IID process $[\alpha_t \,\, \beta_t]$ has  continuous bivariate distribution with a probability density function $f(x,x^*)=\frac{1}{\sigma^2}g(\frac{x}{\sigma},\frac{x^*}{\sigma})$ depending only on one, scale parameter $\sigma >0$, with the density $g(x,y)$ not depending on any parameters, such that $E\big(\frac{\alpha_t^2}{\sigma^2}\big)=E\big(\frac{\beta_t^2}{\sigma^2}\big)=1$. Then, it is easy to show that $\text{ICV}_{\text{AMP}}$ is constant and assumes the form $\text{ICV}_{\text{AMP}}=\frac{C_g}{\sqrt{2-C_g^2}}$, where $C_g=E\bigg(\sqrt{\frac{\alpha_t^2}{\sigma^2}+\frac{\beta_t^2}{\sigma^2}}\bigg)= \int\limits_{\mathbb{R}^2} \sqrt{x_1^2 + x_2^2}g(x_1,x_2) \text{d}x_1 \text{d}x_2$ does not depend on $\sigma$. \textcolor{black}{To sum up, any  $[\alpha_t \,\, \beta_t]$ with distribution depending only on one scale parameter can introduces strong limitations in modeling the amplitude of cyclical fluctuations.}
\end{remark}

In view of the above, one may seek for alternative specifications of the distribution of $[\alpha_t \,\, \beta_t]$ \textcolor{black}{(uncorrelated but not necessarily independent)}. Below, we consider a few such alternatives, presenting the resulting pdf's and inverse coefficients of variation of the amplitude, and discussing their limitations. \textcolor{black}{Cases A-D concern situations that can easily degenerate into the case of the Gaussian distribution for $[\alpha_t \,\, \beta_t]$. The final case (E) discusses statistical properties of the polar coordinates for $[\alpha_t \,\, \beta_t]$ in examples that do not reduce to the Gaussian case.}
\begin{itemize}
    \item[A.] {\it Student's t-distribution.} If $[\alpha_t \,\, \beta_t]$ follows a bivariate Student's \textit{t}-distribution with $\nu>2$ degrees of freedom and a diagonal covariance matrix $\sigma^2 \textbf{I}_2$ (see \cite{kotz2004multivariate}), then $\text{ICV}_{\text{AMP}}=\text{ICV}_{\text{AMP}}(\nu)=\pi  \sqrt{\frac{1}{\frac{2 \pi  (\nu -2) \Gamma \left(\frac{\nu -2}{2}\right)^2}{\Gamma \left(\frac{\nu -1}{2}\right)^2}-\pi ^2}}$ can be shown to be an increasing function of $\nu>2$ (we omit derivation), such that $\lim\limits_{\nu \to 0^{+}}\text{ICV}_{\text{AMP}}(\nu)=0$ and $\lim\limits_{\nu \to \infty}\text{ICV}_{\text{AMP}}(\nu)=\sqrt{\frac{\pi}{4-\pi}}$ and thus still restricting severely the flexibility of both the mean and standard deviation of the amplitude.
    \item[B.]  {\it{Kotz-type elliptical distribution.}} Following \cite{Balakrishnan2009} (see Section 13.6.1 therein, with $\rho=0$ to ensure a zero correlation, \textcolor{black}{although the coordinates remain dependent}), one can easily obtain that $$
    f_{\text{AMP}}(\xi)=2 s \xi ^{2 N-1} r^{N/s} e^{-r \xi ^{2 s}}/\Gamma \left(\frac{N}{s}\right),$$ which is a generalized gamma distribution with the pdf given by $f_{\text{AMP}}(\xi)=\tilde{\gamma } e^{-\left(\frac{\xi -\mu}{\tilde{\beta }}\right)^{\tilde{\gamma }}} \left(\frac{\xi-\mu }{\tilde{\beta }}\right)^{\tilde{\alpha } \tilde{\gamma }} (\xi-\mu)^{-1} / \Gamma \left(\tilde{\alpha }\right)$, with shape parameters $\tilde \alpha = \frac{N}{s}$ and $\tilde \gamma= 2 s$, a scale parameter $\tilde \beta = \left(\frac{1}{r}\right)^{\frac{1}{2 s}}$, and a location parameter $\mu=0$. 
    Straightforward calculations lead to the following expression for the inverse coefficient of variation: 
    \begin{equation}
    \text{ICV}_{\text{AMP}}=\sqrt{\frac{\Gamma \left(\frac{N}{s}+\frac{1}{2 s}\right)^2}{\Gamma \left(\frac{N}{s}+\frac{1}{s}\right) \Gamma \left(\frac{N}{s}\right)-\Gamma \left(\frac{N}{s}+\frac{1}{2 s}\right)^2}}\label{ICV_Kotz}    
    \end{equation}
    and can take any value from $\mathbf{}{R}_{+}$. 
    The formula reveals no strong restrictions between the mean and standard deviation of the amplitude distribution. Despite this apparent advantage, there exist no such approaches in the literature that would yield the Kotz-type elliptical distribution for the amplitude of $\{y_t\}$. 
    Incidentally, also note that under $N=s=1$, the Kotz-type elliptical distribution reduces to the bivariate Gaussian distribution (with $\text{ICV}_{\text{AMP}}$ given by (\ref{relationAMP})). Therefore, the former represents a generalization of the normal case. For an illustration, Figure \ref{fig:1} depicts the pdf of a Kotz-type elliptical distribution under $r=1$, $N=20$, $s=1$.
\begin{figure}[h]
    \centering
    \includegraphics[width=10 cm]{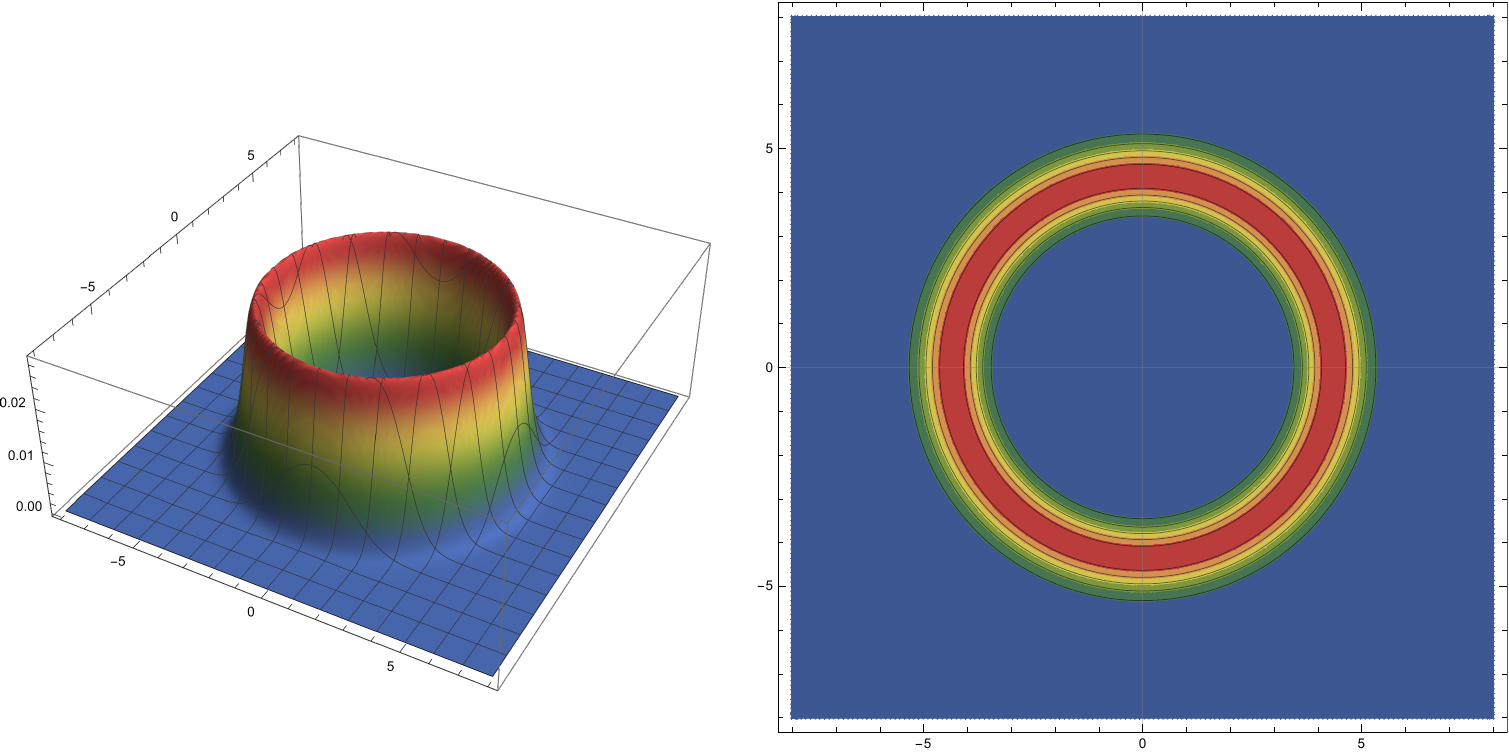}
    \caption{Pdf of the Kotz-type elliptical distribution (\cite{Balakrishnan2009}, Section 13.6.1) with $r=1$, $N=20$, $s=1$.}\label{fig:1}
\end{figure}
    \item[C.] {\it{Gumbel-type elliptical distribution}} (see Section 13.6.3 in \cite{Balakrishnan2009}, with $\rho=0$ to ensure a zero correlation, \textcolor{black}{although the coordinates remain dependent}). The pdf of the Gumbel-type elliptical distribution, under a zero correlation is 
    \begin{equation}
    f(x,x^*)=\frac{a b \exp \left(-a \left(x^2+x^{*2}\right)\right) \exp \left(-b \exp \left(-a \left(x^2+x^{*2}\right)\right)\right)}{\pi  (1-\exp (-b))},    
    \end{equation}
    where $a>0$ and $b>0$ are the scale and shape parameters, respectively, while the pdf of the amplitude takes the form 
\begin{equation}
 f_{\text{AMP}}(\xi)=\frac{2 \pi  a b \xi \exp \left(-b \exp \left(-a \xi^2\right)-a \xi^2\right)}{\pi  (1-\exp (-b))}.  
\end{equation}
Since $a$ accounts for the scale of the distribution, $\text{ICV}_{\text{AMP}}$ depends here only on a single parameter of $b$ (and therefore, restricting the flexibility), and can be shown to be an increasing function thereof (we drop technical details). Moreover, it can also be shown to reduce to $\sqrt{\frac{\pi}{4-\pi}}$ as $b \to 0^{+}$, which results from the fact that the Gumbel-type elliptical distribution generalizes the Gaussian distribution (although in a different manner than the two distributions discussed earlier). It follows that $\text{ICV}_{\text{AMP}}(b)>\sqrt{\frac{\pi}{4-\pi}}$, which means that the ratio of the amplitude's mean to standard deviation is higher than in the normal case. 
\item[D.] {\it{\textcolor{black}{Mixture} of Gaussian distributions.}} 
Conceivably, another possible solution to a highly restrictive Gaussian distribution can be a finite mixture of normal distributions. To that end, let us assume that the joint distribution of $[\alpha_t\,\,\, \beta_t]$ is a mixture of $2^{m+1}$ bivariate Gaussian distributions for $m \in \mathbb{N}$, each featuring the same covariance matrix $\sigma^2 \textbf{\emph{I}}$. On the other hand, the mean vectors $\boldsymbol{\mu}_j=(\mu_{1,j},\mu_{2,j})$, $j=1,2,\ldots,m+1$ are component-specific, but evenly distributed on a circle centered at $(0,0)$ and radius $\mu>0$ (see examples in Figure \ref{fig:2A}). Elementary calculations show that Assumption \ref{ass1} i) holds, which ensures the stationarity of $\{y_t\}$ given by (\ref{proc_yt}). 
Under the Gaussianity, by (\ref{pdfAMP}) and some elementary calculations, we get
    \begin{equation}
    \begin{split}
     f_{\text{AMP}}(\xi) & =  \frac{\xi}{2^{m+1}}\sum\limits_{j=1}^{2^{m+1}} \int _0^{2 \pi }\frac{\exp \left(-\frac{(\xi  \sin (\theta )+\mu_{2,j} )^2}{2 \sigma ^2}-\frac{(\xi  \cos (\theta )+\mu_{1,j} )^2}{2 \sigma ^2}\right)}{2 \pi  \sigma ^2} d\theta\\
     & = \frac{\xi  e^{-\frac{2 \mu ^2+\xi ^2}{2 \sigma ^2}} I_0\left(\frac{\sqrt{2} \mu  \xi }{\sigma ^2}\right)}{\sigma ^2},
     \end{split}    
    \end{equation}
    where $I_n(\cdot)$ denotes the modified Bessel function of the first kind. Note that the distribution of the amplitude here does not depend on the number $m$ of the mixture components. Some tedious algebra leads to the following formula for $\text{ICV}_{\text{AMP}}$ (represented here in terms of the ratio $k=\frac{\mu}{\sigma}$):
    \begin{equation}
     \text{ICV}_{\text{AMP}}(k)= \sqrt{\frac{1}{\frac{4 e^{k^2} \left(k^2+1\right)}{\pi  \left(k^2 I_1\left(\frac{k^2}{2}\right)+\left(k^2+1\right) I_0\left(\frac{k^2}{2}\right)\right)^2}-1}} , 
    \end{equation}
    which can be shown to be increasing with $k \in \mathbb{R}_{+}$ (we skip technicalities). Moreover, $\lim\limits_{k \to 0^{+}} \text{ICV}_{\text{AMP}}(k)=\sqrt{\frac{\pi}{4-\pi}}$, which means that $\text{ICV}_{\text{AMP}}(k)> \sqrt{\frac{\pi}{4-\pi}}$. To conclude, although a mixture of bivariate normals appears to lend some flexibility to the amplitude, the inverse of its coefficient of variation is actually a function of a ratio of two parameters. Moreover, and similarly to the Gumbel-type elliptical distribution, the value of $\text{ICV}_{\text{AMP}}$ is restricted to be higher than $\sqrt{\frac{\pi}{4-\pi}} \approx 1.91$.
    \begin{figure}[ht]
    \centering
    \includegraphics[width=10 cm]{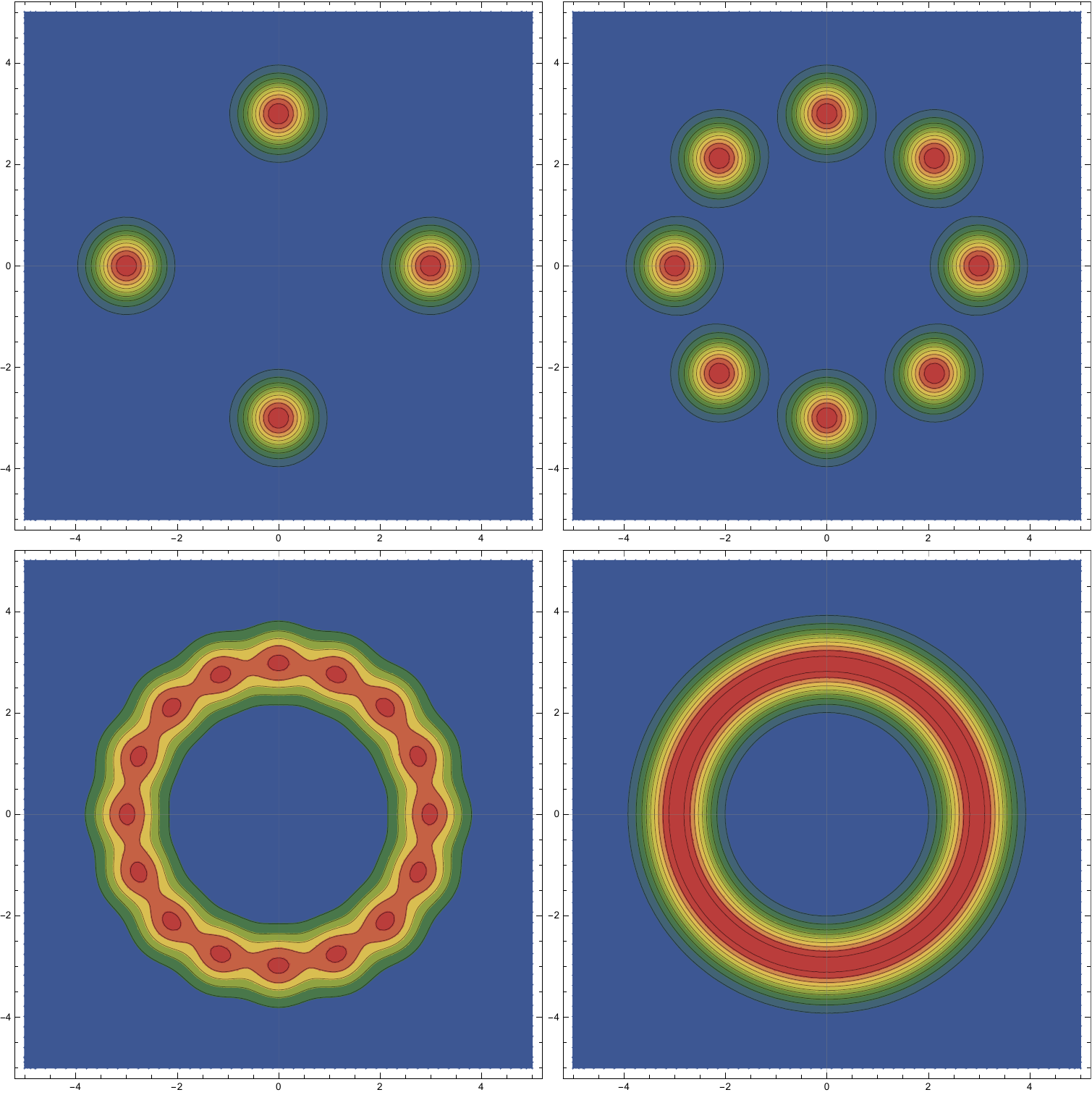}
        \caption{Mixture of bivariate normal distributions with $m=1,2$ (top panel), and $m=3,4$ (bottom panel), radius $\mu=3$, and $\sigma=1/\sqrt{5}$.}\label{fig:2A}
\end{figure}
\item[E.] {\it A proposition of distribution by polar coordinates.} Finally, instead of considering various distributions for $[\alpha_t\,\,\, \beta_t]$ (as done above), we propose to move from the process representation: $y_t=\alpha_t \cos \lambda t + \beta_t \sin \lambda t$, to the equivalent representation: $y_t=A_t \sin (\lambda t + \theta_t)=A_t \sin \theta_t \cos \lambda t + A_t \cos \theta_t \sin \lambda t$, and specify such a distribution for $[A_t\,\,\, \theta_t]$, (where $A_t=\sqrt{\alpha_t^2 + \beta_t^2}$ is the amplitude, and $\theta_t$ is the phase shift) that ensures the stationarity of $y_t$. To that end, we assume that $\theta_t$ follows a $Beta(n,m)$ distribution (for any $t \in \mathbb{Z}$) on the interval $(0,2\pi)$, while $A_t$ is any stationary process. Additionally, we assume that $A_t$ and $\theta_t$ are mutually independent for any $t \in \mathbb{Z}$. To ensure the stationarity of $y_t$, the following set of conditions is necessary (see Theorem \ref{tw21} ii)):
\begin{itemize}
    \item[i)] $E(A_t \sin \theta_t)=E(A_t \cos \theta_t)=0$,
    \item[ii)] $E(A_t^2 \sin^2 \theta_t)=E(A_t^2 \cos^2 \theta_t)$,
    \item[iii)] $E(A_t^2 \sin \theta_t \cos \theta_t)=0$.
\end{itemize}
Elementary analytical calculations shows that i)-iii) hold if and only if $m=n=1$, which gives a uniform distribution for $\theta_t$ on the interval $(0,2\pi)$. Then, the random variables $\cos \theta_t$ and $\sin \theta_t$ follow a $Beta(\frac{1}{2},\frac{1}{2})$ distribution on the interval $(-1,1)$. Below we consider a few examples for $A_t$ (points E.1-E.3).
\begin{itemize}
    \item[E.1.] Assume that $A_t$ is Gaussian process with mean $\mu$ and standard deviation $\sigma$, then (after elementary calculations) the distribution of a bivariate random variable $[\alpha_t \,\, \beta_t]=[A_t \sin \theta_t \,\,\, A_t \cos \theta_t]$ can be shown to have the following pdf:
\begin{equation}
f(x,x^*)=\frac{e^{-\frac{\left(\sqrt{x^2+x^{*2}}-\mu \right)^2}{2 \sigma ^2}}}{2 \pi^{3/2}  \sigma  \sqrt{2x^2+2x^{*2}}}.
\label{PDF:own}
\end{equation}
For an illustration, Figure \ref{fig:2B} presents the pdf for $\mu=4$ and $\sigma=\frac{1}{2}$. 
    \begin{figure}[h!]
    \centering
    \includegraphics[width=10 cm]{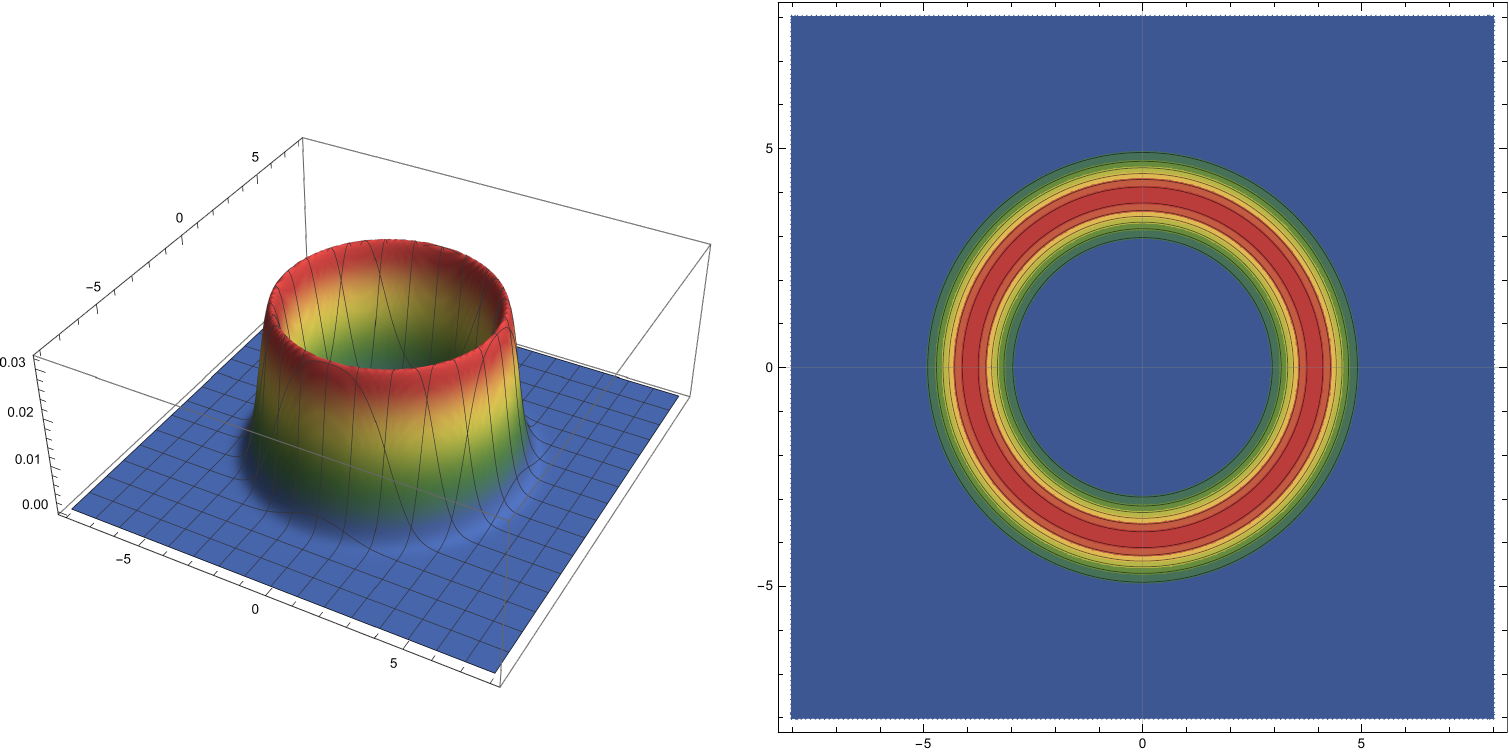}
        \caption{PDF given by (\ref{PDF:own})   with $\mu=4$ and $\sigma=\frac{1}{2}$.}
        \label{fig:2B}
\end{figure}
The amplitude process is $\text{AMP}_t=|A_t|$, which gives the following expression for the inverse coefficient of variation (as a function of $k=\frac{\mu}{\sigma}$): 
$$\text{ICV}_{\text{AMP}}(k)=\sqrt{\frac{\left(k \text{erf}\left(\frac{k}{\sqrt{2}}\right)+\sqrt{\frac{2}{\pi }} e^{-\frac{k^2}{2}}\right)^2}{-\left(k \text{erf}\left(\frac{k}{\sqrt{2}}\right)+\sqrt{\frac{2}{\pi }} e^{-\frac{k^2}{2}}\right)^2+k^2+1}},$$ 
where $\text{erf}(z)=\frac{2}{\sqrt{\pi}}\int\limits_{0}^{z}e^{-t^2}dt$ is the error function. As a function of $k$, the above expression is an continuous function on the real line, increasing over $(0,\infty)$ and decreasing over $(-\infty,0)$. Moreover, $\lim\limits_{k \to 0^{}}\text{ICV}_{\text{AMP}}(k)= \sqrt{\frac{2}{\pi -2}} \approx 1.32$, which is lower that in cases C and D.
\item[E.2.] Incidentally, notice that to ensure that the amplitude is positive, one can specify it as $A_t=e^{B_t}$, where $B_t$ is a Gaussian stationary process with an unconditional mean $\mu$ and standard deviation $\sigma$. Therefore, $A_t$ follows a log-normal distribution, for which the inverse coefficient of variation is simply $\text{ICV}_{\text{AMP}}=1/\sqrt{e^{\sigma^2}-1}$, admitting any value in $R_{+}$.
\item[E.3.] To generalize, $A_t$ can in fact be specified as any stochastic process with such a time-invariant unconditional distribution defined on $R_{+}$ that its inverse coefficient of variation admits any value in $R_{+}$ (it is straightforward to check that it is the case for, e.g., a gamma distribution, inverse gamma distribution, Nakagami distribution, and others).
\end{itemize}  
\end{itemize}
\vspace{1 cm}
\textcolor{black}{
In conclusion, assuming a bivariate uncorrelated Gaussian distribution for $[\alpha_t\,\,\, \beta_t]$ is patently a deficient approach. This result also holds for any dependent bivariate distributions featuring only a singe scale parameter. 
As shown in the examples A-D, the problem arises from the distribution of the $[\alpha_t \,\, \beta_t]$ variable and the resulting properties of the amplitude, which do not prove sufficient in terms of their flexibility. A sole exception to this a somewhat gloom perspective is the Kotz-type elliptical distribution (Point B). Although the distribution itself is of quite a limited popularity, and the formula of $\text{ICV}_{\text{AMP}}$ exhibits some complexity (see (\ref{ICV_Kotz})), the latter can take any value from $R_{+}$, and thus does not suffer from the restraints afflicting the other distributions.
The last example (Point E) takes a different and our novel approach, where the distributional assumptions are made in terms of the polar coordinates instead of $[\alpha_t \,\, \beta_t]$. Owing to that, the amplitude's inverse coefficient of variation can  lower than in Cases C and D or take any positive real value. At the end of this part, we emphasize that in this work, we construct a cyclic process for which amplitude process and phase shift distributions are defined directly and hence such characteristics as analyzed here inverse-coefficient of variation for amplitude process can take any positive value.}

\subsection{Some results related to multi-frequency case}\label{sub:22}
In this subsection we shift our focus to the multi-frequency framework, with a set of frequencies $0<\lambda_1 \leq \lambda_2 \leq \ldots \leq \lambda_{K} \leq \pi$. The theorem below provides explicit formulae for the autocovariance function and power spectral density of a such cyclic process. 

\begin{theorem}
Assume that for any $j=1,2,\ldots,K$, the bivariate random processes $[\alpha_{t,j} \,\, \beta_{t,j}]$ are mutually uncorrelated and have a covariance matrix of the form  $E([\alpha_{t+\tau,j} \,\, \beta_{t+\tau,j}]'[\alpha_{t,j} \,\, \beta_{t,j}])=\gamma_j(\tau) \textbf{\emph{I}}_{2}$, and a diagonal power spectral density matrix of the form $\boldsymbol{f}_j(\omega)=f_{j}(\omega)\textbf{\emph{I}}_{2}$. Then, for any $0<\lambda_1 \leq \lambda_2 \leq \ldots \leq \lambda_{K} \leq \pi$, the zero-mean stationary process $\{y_t\}$ defined as 
\begin{equation}
y_t=\sum\limits_{j=1}^{K}  (\alpha_{t,j} \cos \lambda_j t + \beta_{t,j} \sin \lambda_j t),   
\end{equation}
has the autocovariance function 
\begin{equation}
\gamma_y(\tau)=\sum\limits_{j=1}^{K}\gamma_j(\tau) \cos \lambda_j \tau,    
\end{equation} 
and the power spectral density function
\begin{equation}
f_{y}(\omega)=\frac{1}{2}\sum\limits_{j=1}^{K}( f_j(\omega-\lambda_j) + f_j(\omega+\lambda_j)).
\label{eq:psd}
\end{equation} 
\label{tw:psd}
\end{theorem}

The above theorem generalizes results presented in \cite{Proietti2023}, where the autocovariance function and power spectral density were derived in a very special case that we discuss in more detail in Section 3 (see also some introductory remarks in Section 1). 

\begin{remark}
The power spectral mass concentration under certain frequencies $\lambda_j$ in (\ref{eq:psd}) can be boosted (see \cite{HarveyTrimbur2003} and \cite{Trimbur2006} in the context of the {\it{nth-order stochastic cycle}}). To adopt the idea into our framework, it is enough to take  
$\alpha_{t,j}=\tilde \alpha_{t,j}^{(n)}$ and $\beta_{t,j}=\tilde \beta_{t,j}^{(n)}$, where $\{\tilde \alpha_{t,j}^{(n)}\}$ and $\{\tilde \beta_{t,j}^{(n)}\}$ are defined recursively as 
\begin{equation}
\tilde \alpha_{t,j}^{(s)}=\sum\limits_{i=0}^{\infty} \psi_i \alpha_{t-i,j}^{(s-1)} \,\,\,\, \text{and} \,\,\,\,\,\, \tilde \beta_{t,j}^{(s)}=\sum\limits_{i=0}^{\infty} \psi_i \beta_{t-i,j}^{(s-1)},   
\end{equation} 
for $s=2,3,\ldots,n$, with 
$\tilde \alpha_{t,j}^{(1)}=\sum\limits_{i=0}^{\infty} \psi_i \epsilon_{t-i,j}$ and $\tilde \beta_{t,j}^{(1)}=\sum\limits_{i=0}^{\infty} \psi_i \epsilon_{t-i,j}^{*}$, 
where $\{\epsilon_{t,j}\}$ and $\{\epsilon_{t,j}^*\}$ are uncorrelated white noises sharing the same variance $\sigma_{j}$. Finally, by $f_{1,j}(\omega)$ we denote the power spectral density function of $\{\tilde \alpha_{t,j}^{(1)}\}$ and $\{\tilde \beta_{t,j}^{(1)}\}$. Then, based on the theorem above, and using some elementary calculations, we get that the power spectral density of $\{y_t\}$ for $\alpha_{t,j}=\tilde \alpha_{t,j}^{(n)}$ and $\beta_{t,j}=\tilde \beta_{t,j}^{(n)}$ has the form 
\begin{equation}
f_{y}(\omega)=\frac{1}{2}\sum\limits_{j=1}^{K}( f_{1,j}^{n}(\omega-\lambda_j) + f_{1,j}^{n}(\omega+\lambda_j)).
\label{psd:smooth}
\end{equation}
\end{remark}

We continue with a theorem concerning the representation of any zero-mean stationary process in the form
\begin{equation}
y_{t} = \sum\limits_{j=1}^{K} \alpha_{j,t} \cos \lambda_j t +   \beta_{j,t} \sin \lambda_j t,  
\label{eq:yt0}
\end{equation}
with $[\alpha_{j,t}\,\,\, \beta_{j,t}]$ being bivariate processes mutually independent for $j=1,2,\ldots,K$. This indicates the ambiguity of defining the amplitude of fluctuations also for many frequencies.

\begin{theorem}
Let $\{y_{t}:\, t \in \mathbb{Z}\}$ be any zero-mean stationary time series with an $\text{MA}(\infty)$ representation \begin{equation}
y_{t}=\Psi(L)\epsilon_{t},     
\end{equation} where $\Psi(L)=\sum\limits_{j=0}^\infty \psi_j L^j$, $\sum\limits_{j=0}^{\infty}|\psi_j| <\infty$ and where $\{\epsilon_{t}\}$ is a sequence of zero-mean IID  random variables which distribution being a marginal distribution for some bivariate spherical distribution with characteristic function $\phi(\|\textbf{z}\|)$ at $\textbf{z} \in \mathbb{R}^2 $. Then, for any frequencies $0<\lambda_1 \leq \lambda_2 \leq \ldots \leq  \lambda_{K} \leq \pi$, the process $\{y_t\}$ can be represented in infinitely many ways as 
\begin{equation}
y_{t} = \sum\limits_{j=1}^{K} \alpha_{j,t} \cos \lambda_j t +   \beta_{j,t} \sin \lambda_j t,  
\label{eq:yt}
\end{equation}
where $[\alpha_{j,t}\, \beta_{j,t}]$  are zero-mean bivariate  processes, mutually independent for $j=1,2,\ldots,K$, and having  spherical distributions independent of $t$. 
%
\label{tw:representation}
\end{theorem}

We close this subsection with a theorem establishing a vector autoregression (VAR) representation of the cyclical process $\{y_t\}$ and stationarity up to any order, allowing for dependencies between cyclic components. 

\begin{theorem}
Let $y_t=\sum\limits_{j=1}^{K}z_{j,t}$, where $z_{j,t}=\alpha_{j,t} \cos \lambda_j t + \beta_{j,t} \sin \lambda_j t$, and $0 <  \lambda_1< \lambda_2<\ldots,\lambda_K\leq \pi$. Assume that for any $j=1,2,\ldots,K$, the processes $\{\alpha_{j,t}\}$ and $\{\beta_{j,t}\}$ are stationary autoregressive processes of order $p_j \in \mathbb{N}$ with the same characteristic polynomial $\Phi_j(L)=1-\sum\limits_{i=1}^{p_j}\phi_{j,i}L^i$, that is
$$\left\{\begin{array}{lcl}
\Phi_j(L)\alpha_{j,t}&=&\epsilon_{j,t}\\
\Phi_j(L)\beta_{j,t}&=&\epsilon_{j,t}^*,
\end{array}\right.$$
 where $[\epsilon_{j,t}\,\,\, \epsilon_{j,t}^*]$ is IID bivariate random sequence with some spherical distribution. Then
 \begin{itemize}
 \item[i)] for any $j=1,2,\ldots,K$, the process $\{z_{j,t}\}$ is a coordinate of a bivariate $\text{VAR}(p_j)$ process of the form 
\begin{equation}
 \left[
\begin{array}{c}
 z_{j,t} \\
 z_{j,t}^{*} \\
\end{array}
\right]=  \sum\limits_{i=1}^{p_j} \left(\phi_{j,i} \textbf{\emph{R}}(i \lambda_j)  \left[
\begin{array}{c}
 z_{j,t-i} \\
 z_{j,t-i}^{*} 
\end{array}
\right] \right)  +    \textbf{\emph{R}}(t \lambda_j)  \left[
\begin{array}{c}
 \epsilon_{j,t} \\
 \epsilon_{j,t}^{*} 
\end{array}
\right];
\end{equation}
\item[ii)] \textcolor{black}{assuming that the joint distribution of $\boldsymbol{\epsilon}_{t}=[\epsilon_{1,j}\,\, \epsilon_{1,t}^*\,\,\epsilon_{2,j}\,\, \epsilon_{2,t}^* \,\,\ldots \,\, \epsilon_{K,j}\,\, \epsilon_{K,t}^*]$ has a pdf at point $(x_1,y_1,x_2,y_2,\ldots,x_K,y_K) \in \mathbb{R}^{2K}$ given by $g(x_1^2+y_1^2,x_2^2+y_2^2,\ldots,x_K^2+y_K^2)$, where $g:[0,\infty)^K \to \mathbb{R}$ and that for some positive integer $m$ we have that $E|\epsilon_{j,t}|^m<\infty$,  for all $j=1,2,\ldots,K$, then $\{y_t\}$ is stationary up to order $m$. } 
\end{itemize}  
\label{Hannan_to_Harvey}
\end{theorem}


\subsection{The concept of strongly pseudo-cyclical autocovariance function}

In the literature related to stationary processes, a zero-mean stationary time series $\{y_t:\,t\in \mathbb{Z}\}$ displaying cyclic features (with either one frequency $\lambda$ or a set of frequencies $\boldsymbol{\lambda}=(\lambda_1,\lambda_2,\ldots,\lambda_K) \in (0,\pi]^K$) is typically referred to as {\it cyclical} (e.g., {\it a cyclical long memory process}, {\it cyclical component}, {\it cyclical pattern}, or a {\it cyclical model}; see, e.g., \cite{Arteche2000}, \cite{Proietti2023}, \cite{Hyndman2008}). Another common term of reference is a process with {\it 'pseudocyclical'} behaviour (see \cite{Trimbur2006}). Nevertheless, the definition of a cyclical property in the zero-mean stationary case has not been explicitly linked to characteristics such as the autocovariance or spectral density function. Below, we formulate a formal definition in relation to the autocovariance function, through the concept of a {\it strongly pseudocyclical autocovariance function}.

\begin{definition}
We say that a zero-mean weakly stationary process $\{y_t:\,t\in \mathbb{Z}\}$ has {\it strongly pseudo-cyclical autocovariance function} with $K$ different frequencies: $\lambda_1, \lambda_2, \ldots, \lambda_K \in (0, \pi]$ if its autocovariance function satisfies
\begin{equation}
E(y_t y_{t+\tau})=\sum\limits_{j=1}^{K}a_{j,\tau} \cos \lambda_j \tau + b_{j,\tau} \sin \lambda_j \tau ,  
\label{pseudo-cyclical}
\end{equation} where the sequences $\{a_{j,\tau}\}_{\tau \in \mathbb{Z}}$ and $\{b_{j,\tau}\}_{\tau \in \mathbb{Z}}$ are such that  for any $j=1,2,\ldots,K$ $$\lim\limits_{|\tau| \to \infty }a_{j,\tau}=\lim\limits_{|\tau| \to \infty }b_{j,\tau}=0$$ and there exists $\tau_0 >0$ such that both $|a_{j,\tau}|$ and $|b_{j,\tau}|$ are a non-increasing function at $\tau>\tau_0$.
\label{def:21}
\end{definition}

To illustrate the above-stated definition, let us consider a simple process of the form:
\begin{equation}
y_t=\alpha_t \cos \omega_1 t + \beta_t \sin \omega_1 t,    
\end{equation}
where $\alpha_t= \xi_t \cos \omega_2 t + \xi_t' \sin \omega_2 t$ and $\beta_t= \eta_t \cos \omega_2 t + \eta_t' \sin \omega_2 t$, $\omega_1,\omega_2 \in (0,\pi]$ and
 $\xi_t=\rho \xi_{t-1} + \kappa_t$, $\xi_t'=\rho \xi_{t-1}' + \kappa_t'$,  $\eta_t=\rho \eta_{t-1} + \epsilon_t$, $\eta_t'=\rho \eta_{t-1}' + \epsilon_t'$, $|\rho|<1$. For  
$[\kappa_t\,\, \kappa_t'\,\,\, \epsilon_t \,\,\, \epsilon_t']$ we assume that follows a Gaussian white noise with a covariance matrix $\sigma^2 \textbf{I}_4$. Hence, by Theorem \ref{tw:psd}, the autocovariance functions of the processes $\alpha_t$ and $\beta_t$ are the same:  $\gamma(\tau)=\frac{\rho^{|\tau|}\sigma^2}{1-\rho^2}\cos \omega_2 \tau $. The same theorem yields also the autocovariance function of $y_t$:
$\gamma_y(\tau)=\frac{\rho^{|\tau|}\sigma^2}{1-\rho^2}\cos \omega_2 \tau  \cos \omega_1 \tau =\frac{1}{2}\frac{\rho^{|\tau|}\sigma^2}{1-\rho^2} \cos (\omega_1+\omega_2)\tau + \frac{1}{2}\frac{\rho^{|\tau|}\sigma^2}{1-\rho^2} \cos |\omega_1-\omega_2|\tau$, the formula of which reveals the relation of this autocovariance function to two (unless coinciding) frequencies $\lambda_1,\lambda_2 \in (0,\pi]$ related to $\omega_1 + \omega_2$ and $\omega_1 - \omega_2$. Specifically, $\lambda_2=|\omega_1-\omega_2|$, with $a_{2,\tau}=\frac{1}{2}\frac{\rho^{|\tau|}\sigma^2}{1-\rho^2}$ and $b_{2,\tau}=0$. If $\omega_1+\omega_2 \leq \pi$, then $\lambda_1=\omega_1+\omega_2 \in (0,\pi]$, with $a_{1,\tau}=\frac{1}{2}\frac{\rho^{|\tau|}\sigma^2}{1-\rho^2}$ and $b_{1,\tau}=0$. Otherwise, that is if $\omega_1+\omega_2 > \pi$, then $\lambda_1=\omega_1+\omega_2 -\pi \in (0,\pi]$, with $a_{1,\tau}=-\frac{1}{2}\frac{\rho^{|\tau|}\sigma^2}{1-\rho^2}$ and $b_{1,\tau}=0$.


In the next section, we will verify for which cyclical processes the properties of the strongly pseudo-cyclical autocovariance function match the definition of \ref{def:21}. 

\section{Review of existing cyclical processes}

To the best of our  knowledge, the literature on cyclical zero-mean stationary processes is scarce, and few such processes have been proposed so far. In this section we review them and examine briefly their properties, including also (whenever possible to show) their strongly pseudo-cyclical autocovariance functions (according to Definition \ref{def:21}). The review covers Hannan's model (see \cite{Hannan1970}), stochastic cycle model (see \cite{Harvey1985}), nth-order stochastic cycle (see \cite{HarveyTrimbur2003}), elliptical and spherical stochastic cycles (see \cite{LuatiProietti2010}), k-factor GARMA model (\cite{GARMA1998}), and fractional sinusoidal waveform process (see \cite{Proietti2023}).

\subsection{Hannan's model and its generalization} In \cite{Hannan1970}, the following model with time-varying amplitude and phase shift was proposed for seasonal time series:  
\begin{equation}
 y_t=\sum\limits_{j=1}^{6}(\alpha_{j,t}\cos(t \lambda_j) +\beta_{j,t} \sin(t \lambda_j) ), 
 \label{Hannan1970}
\end{equation}
where $\lambda_j=2\pi j /12$ and $\alpha_{j,t}  =  \rho_{j} \alpha_{j,t-1} + \eta_{j,t}$, $\beta_{j,t}  =  \rho_{j} \beta_{j,t-1} + \tilde \eta_{j,t}$,  
with $[\eta_{j,t}\,\tilde \eta_{j,t}]$ being a Gaussian white noise with the covariance matrix $\sigma_j^2 \textbf{I}_2$. Equation (\ref{Hannan1970}) can be generalized to any number $K\in \mathbb{N}$ of cyclical components $\{z_{j,t}\}$, each related to the corresponding frequency $\lambda_j \in (0,\pi]$, $j=1,2,\ldots,K$, with $ 0 < \lambda_1< \lambda_2<\ldots,\lambda_K\leq \pi$: 
\begin{equation}
 y_t=\sum\limits_{j=1}^{K}\overbrace{(\alpha_{j,t}\cos(t \lambda_j) +\beta_{j,t} \sin(t \lambda_j) )}^{z_{j,t}},
 \label{GeneralizationHannan1970}
\end{equation}
with $\alpha_{j,t}$ and $\beta_{j,t}$ being mutually independent and stationary Gaussian processes sharing the same autocovariance function $\gamma_j(\tau)$. Our Theorem \ref{tw:psd} provides the autocovariance function (which is strongly pseudo-cyclical here) and power spectral density function of $\{y_t\}$ given by  (\ref{GeneralizationHannan1970}).  Under the Gaussianity assumption, for each component $z_{j,t}$, the related amplitude features the inverse coefficient of variation of the form (\ref{ICV_Gaussian}), which limits possible applications. 
\subsection{Stochastic cycle model} After the work of \cite{Hannan1970}, an equivalent model to (\ref{GeneralizationHannan1970}) for a univariate time series $\{y_t\}$ was introduced and developed by \cite{Harvey1985}, \cite{Harvey1989}, \cite{HarveyJaeger1993}. In the cited works, this equivalent formulation of each $z_{j,t}$ in (\ref{GeneralizationHannan1970}) is referred to as a {\it stochastic cycle} $\{c_{j,t}\}$ (related to the frequency $\lambda_j$), and it has an $\text{ARMA}(2,1)$ representation of the form: 
\begin{equation}
(1-2\tilde \rho_j \cos \lambda_j L + \tilde \rho_j^2 L^2) c_{j,t}= \kappa_{j,t} - \tilde \rho_j \cos \lambda_j \kappa_{j,t-1} + \tilde \rho_j \sin \lambda_j \kappa_{j,t-1}^*, 
\end{equation}
where $|\tilde \rho_j|<1$, $L$ denotes the backshift operator (i.e., for some stochastic process $x_t$ and any integer $k$ we have $x_t L^k = x_{t-k}$), and $[\kappa_{j,t} \, \kappa_{j,t}^*]$ is a Gaussian white noise with a covariance matrix $\sigma_{\kappa_j}^2 \textbf{I}_2$. Furthermore, the stochastic cycle $\{c_{j,t}\}$ can be represented as an coordinate of a relevant bivariate VAR($1$) process, with its $2 \times 2$ autoregressive matrix being a rotation matrix, which follows easily from our Theorem \ref{Hannan_to_Harvey} under $K=1$ and $p_1=1$. Based on the same theorem, we conclude that the stochastic cycle, although largely popularized and commonly employed in business cycle analyses, does not bring anything new and is equivalent to a simple generalization of Hannan's model, given by (\ref{GeneralizationHannan1970}) under $K=1$. Hence, the relevant inverse coefficient of variation continues to share the Gaussian case restraint (see (\ref{ICV_Gaussian})).
\subsection{nth-order stochastic cycle}
To obtain a 'smoother' stochastic cycle (by stronger mass concentration around certain frequencies for power spectral density function), the concept of an {\it nth-order stochastic cycle} was introduced in \cite{HarveyTrimbur2003} (and later developed by \cite{Trimbur2006}). This generalized cycle $\{\psi_{n,t}\}$ admits the following $\text{ARMA}(2n,2n-1)$ representation:
\begin{equation}
(1-2\rho \cos \lambda L + \rho^2 L^2)^n \psi_{n,t}= \sum\limits_{j=0}^n A_j(L) (\kappa_t \cos \lambda j  -\kappa_t^*  \sin \lambda j ), 
\label{ARMA_rep_n}
\end{equation}
where $|\rho|<1$, $A_{j}(L)=L^{n-1+j} \binom{n}{j}(-\rho)^j$, and $[\kappa_t \,\,\, \kappa_t^*]$ is a Gaussian white noise with a covariance matrix $\sigma_{\kappa}^2 \boldsymbol{I}_2$. The exact form of the autocovariance function was shown in \cite{Trimbur2006}: 
\begin{equation}
E(\psi_{n,t}\psi_{n,t+\tau})=a_{\tau}(\rho) \cos \lambda \tau, 
\end{equation}
where $\lim\limits_{|\tau|\to \infty}a_{\tau}(\rho)=0$. The spectral density function follows immediately from (\ref{ARMA_rep_n}). The amplitude of the nth-order stochastic cycle is $\sqrt{\psi_{n,t}^2+\psi_{n,t}^{*2}}$, where $\psi_{n,t}$ and $\psi_{n,t}^{*}$ are zero-mean Gaussian $\text{ARMA}(2n,2n-1)$ processes with the same parameters and mutually independent Gaussian white noises (see Section 3 in \cite{Trimbur2006}). Hence, again, the inverse coefficient of variation of the amplitude suffers from the Gaussian case restraint (see (\ref{ICV_Gaussian})). Therefore, arguably, although the nth-order stochastic cycle can result in a `smoother' stochastic cycle (for an increasing $n$), and thus a 'smoother' amplitude process, the expectation of the latter equals the amplitude's standard deviation times $\sqrt{\frac{\pi}{4-\pi}}$ (see (\ref{ICV_Gaussian})), which is the very same strong limitation as for the ($n=1$) stochastic cycle.   
\subsection{Elliptical and  spherical stochastic cycles}
The stochastic cycle and nth-order stochastic cycle are related to a VAR representation, where the autoregression (or, transition) matrix is associated with rotation along a circle in the
plane. In \cite{LuatiProietti2010}, the original transition matrix was replaced by the motion of a point along an ellipse. Still, the process remains Gaussian and admits an $\text{ARMA}(2,1)$ representation. Even though it introduces a different structure of the ARMA model (than the nth-order stochastic cycle), by the representation given by Theorem \ref{tw:representation}, we get again a linear relation between the expectation and standard deviation of the amplitude. 
\subsection{$k$-factor GARMA model}
In \cite{GARMA1998} and \cite{Smallwood2003}, the following stationary model of a long-memory $k$-factor GARMA (Gegenbauer ARMA) specification was investigated: 
\begin{equation}
\psi(L)\prod\limits_{j=1}^{K} (1-2 u_i L + L^2)^{c_i} y_t = \theta(L)\epsilon_t,    
\end{equation}
with frequencies $f_i=\frac{1}{2 \pi \cos u_i}$, $j=1, ..., K$. For a GARMA process, the exact form of its spectral density can be derived. However, one of the main limitations is the lack of a closed form expression for the autocovariance function (although it seems likely that it admits a pseudo-cyclical autocovariance function of the form (\ref{pseudo-cyclical})). Despite the fact that the $k$-factor GARMA model has not been defined on the basis of the representation (\ref{eq:yt0}), it is still Gaussian with an $\text{MA}(\infty)$ representation and hence, drawing on Theorem \ref{tw:representation}, for each of the frequencies $0<f_1<f_2<\ldots<f_K \leq \pi$, the relevant cyclic component can be represented in such a way that again, a linear relation between the amplitude's expectation and standard deviation emerges.

\subsection{Fractional sinusoidal waveform process} 
Motivated by the $k$-factor GARMA model, the following fractional sinusoidal waveform process was proposed in \cite{Proietti2023} (see also earlier work:  \cite{maddanu2022modelling}):
\begin{equation}
y_t=\alpha_{t} \cos(\lambda t) + \beta_{t} \sin(\lambda t),
\label{Proietti2023}
\end{equation} 
where $\alpha_t =(1-L)^{-d}\kappa_t$ and $\beta_t=(1-L)^{-d} \kappa_t^{*}$, with $0<d<1/2$ and $[\kappa_t \, \,\, \kappa_t^*]$ being a Gaussian white noise with a covariance matrix $\sigma^2_{\kappa} \boldsymbol{I}_2$. Note that although the theoretical properties of the model were presented in the cited paper for the single-frequency case, a multiple-frequency model (obtained straightforwardly as a sum of such independent components) was considered in the empirical part of the study.

The autocovariance function of the model (\ref{Proietti2023}) assumes the form:
\begin{equation}
E(y_t y_{t+\tau})=\sigma_{\kappa}^2 \frac{\Gamma(1-2d)\Gamma(d+\tau)}{\Gamma(1+\tau-d)\Gamma(1-d)}\cos \lambda \tau,
\end{equation}
which is strongly pseudo-cyclical. The spectral density function features a vertical asymptote at the frequency $\lambda$, was evaluated in the cited work and can be also obtained using our general Theorem \ref{tw:psd}. The linear relation between expectation and standard deviation of amplitude process is obvious here.






\section{New stochastic cycle concept and its statistical properties}

In this section, we present a novel specification of a zero-mean cyclic stationary process, already mentioned in Section 1.4. We focus here on the single-frequency model, since generalizing it into the multi-frequency case (as the sum of independent components, each corresponding to a given frequency)  is straightforward.

For convenience, let us remind the equation of our model proposed by equation~(\ref{model_prep}):
$$y_t=(a+A_t)\sin[\lambda (t+ \psi + P_t)].$$
In what follows we assume that $\psi =0$ (or the sake of parameter identification), and that the amplitude process $A_t$ is stationary up to order $m$. On the other hand, the phase shift process $P_t$ can be considered either stationary (resulting in oscillations around a constant cycle length) or non-stationary (specifically, integrated of order one). Before we formulate two relevant theorems, handling the two cases, respectively, let us define that a process $\{y_t\}$ is an Almost Periodically Correlated up to order $m$ (henceforth denoted as APC($m$)) if for any $s \leq m$, and $\tau_1,\tau_2,\ldots,\tau_s \in \mathbb{Z}$, and $t \in \mathbb{Z}$, the moment $E(y_{t+\tau_1}y_{t+\tau_2}\ldots y_{t+\tau_s})$ exists and is an almost periodic function of $t$. Obviously, the APC($m$) class contains the class of stationary time series up to order $m$.



\begin{theorem}
Let $\{A_t:\, t\in \mathbb{Z}\}$ be a time series with zero mean and stationary up to order $m$, and let $\{P_t:\, t \in \mathbb{Z}\}$ be a stationary Gaussian time series. Assume also that the processes $\{A_t:\, t\in \mathbb{Z}\}$ and $\{P_t:\, t \in \mathbb{Z}\}$ are independent. Then, for any $\lambda \in (0,\pi]$, the process $\{y_t: t \in \mathbb{Z}\}$ of the form 
$$y_t=(a+A_t)\sin[\lambda (t+ P_t)]$$
is an APC($m$) time series, such that for any positive integer $s \leq m$ and any sequence of integers $\tau_1 \leq \tau_2 \leq \ldots \leq \tau_s$, we have 
\begin{equation}
E\left(\prod\limits_{j=1}^s y_{t+\tau_j}\right)=\frac{(-1)^{\lfloor \frac{s}{2} \rfloor}}{2^s} a_{\boldsymbol{\tau}}
  \sum\limits_{(e_1,e_2,\ldots,e_s) \in \boldsymbol{P}} \!\!\!e^{-\frac{1}{2}\text{Var}\left(\lambda \sum\limits_{j=1}^s e_j P_{t+\tau_j}\right)} f\left(  \lambda \sum\limits_{j=1}^s  e_j (t+\tau_j) \right)\prod\limits_{j=1}^{s}e_j,
  \nonumber
\end{equation} 
where $\boldsymbol{P}=\{-1,1\}^s$, $f$ is the sine (or cosine) function if $s$ is odd (even, respectively), and 
 \begin{equation}
  a_{\boldsymbol{\tau}} = E\left( \prod\limits_{j=1}^s(a+A_{t+\tau_j}) \right) .  
 \end{equation}
\label{tw1}
\end{theorem}

The theorem above shows the exact form of all moments up to $m$, which are almost periodic functions of time $t \in \mathbb{Z}$. Therefore, any spectral characteristics related with the second and higher-order spectra for $\{y_t: t \in \mathbb{Z}\}$ can be evaluated. However, since $\{y_t: t \in \mathbb{Z}\}$ is as an APC process and thus nonstationary (which may be somewhat surprising under the assumed stationarity of both $A_t$ and $P_t$), we ditch that specification from further analysis, and focus on the following case, featuring a nonstationary $P_t$, which (again, to one's surprise) yields a stationary $\{y_t\}$, instead.

\begin{theorem}
Let $\{A_t:\, t\in \mathbb{Z}\}$ be a time series with zero mean and stationary up to order $m$, and let $\{P_t:\, t \in \mathbb{Z}\}$ follow an $\text{ARIMA}(p,1,q)$ model with a Gaussian white noise. Assume also that the processes $\{A_t:\, t\in \mathbb{Z}\}$ and $\{P_t:\, t \in \mathbb{Z}\}$ are independent. Then, for any $\lambda \in (0,\pi]$, the process $\{y_t: t \in \mathbb{Z}\}$ of the form 
$$y_t=(a+A_t)\sin[\lambda (t+ P_t)]$$
is zero-mean and stationary up to order $m$, such that for any positive integer $s>1$ and any sequence of integers $\tau_1 \leq \tau_2 \leq \ldots \leq \tau_s$, we have $E\left(\prod\limits_{j=1}^s y_{t+\tau_j}\right)=0$ if $s$ is odd and 
\begin{equation}
E\left(\prod\limits_{j=1}^s y_{t+\tau_j}\right)=\frac{(-1)^{\lfloor \frac{s}{2} \rfloor}}{2^s} a_{\boldsymbol{\tau}} \!\!\!\!\!\!
  \sum\limits_{(e_1,e_2,\ldots,e_s) \in \boldsymbol{P}_0} \!\!\!\!\!\!e^{-\frac{1}{2}\text{Var}\left(\lambda \sum\limits_{j=1}^s e_j P_{t+\tau_j}\right)} \cos\left(  \lambda \sum\limits_{j=1}^s  e_j \tau_j \right) \prod\limits_{j=1}^{s}e_j
\end{equation} 
 if $s$ is even, where $P_0=\{(e_1,e_2,\ldots,e_s)\in \{-1,1\}^s: \sum\limits_{j=1}^s e_j=0\}$ and 
 \begin{equation}
  a_{\boldsymbol{\tau}} = E\left( \prod\limits_{j=1}^s(a+A_{t+\tau_j}) \right) .  
 \end{equation}
\label{tw2}
\end{theorem}
It follows from Theorem~\ref{tw2} that the skewness (if exists) of the process $y_t$ is zero, \textcolor{black}{which is a somewhat surprising result here.  What is more,} for any positive integer $k$, we have (upon existence)
\begin{equation}
  E(y_t^{2k}) = E\left( (a+A_{t})^{2k} \right),
 \end{equation}
which gives the kurtosis of $y_t$:
\begin{equation}
  \text{Kurt}(y_t) = \frac{E\left( (a+A_{t})^{4} \right)}{\left(E\left( (a+A_{t})^{2} \right)\right)^2}.
 \end{equation}
Under an assumption that $A_t$ is Gaussian with variance $\sigma_{A}^2$, the kurtosis simplifies to 
\begin{equation}
  \text{Kurt}(y_t) = 3-\frac{2 a^4}{\left(a^2+\sigma _A^2\right){}^2}=3-\frac{2}{\left(1+\left(\frac{\sigma _A}{a}\right)^2\right){}^2},
 \end{equation}
 which is lower than for the normal distribution, is an increasing function of the argument $\frac{\sigma _A}{a}$, and is bounded from below by $1$.
 
Theorem~\ref{tw2} enables some special cases to consider. For example, if $\{P_t \}$ follows  $\text{ARIMA}(0,1,0)$ (i.e., a random walk: $P_t=P_{t-1} + \epsilon_t$, where $\epsilon_t$ is a zero-mean Gaussian white noise with a variance $\sigma^2$), and $A_t$ is stationary with an autocovariance function $\gamma_{A}(\tau)=\text{Cov}(A_t,A_{t+\tau})$, then for $s=2$ we get $P_0=\{(-1,1),(1,-1)\}$. Hence, for any $(e_1,e_2) \in P_0$, and any $\tau_1 \leq \tau_2$, one obtains $$\text{Var}\left(\lambda \sum\limits_{j=1}^2 e_j P_{t+\tau_j}\right)=\lambda^2(\tau_2 - \tau_1) \sigma^2$$ and $$\cos\left(  \lambda \sum\limits_{j=1}^2  e_j \tau_j \right)=\cos\left(  \lambda (\tau_2-\tau_1) \right).$$ In consequence, for any $\tau \in \mathbb{Z}$,
\begin{equation}
 E(y_t y_{t+\tau})= 2 (a^2 + \gamma_{A}(\tau)) e^{-\frac{1}{2}\lambda^2|\tau|\sigma^2} \cos(\lambda |\tau|).   
\end{equation}



\section*{Appendix}

\begin{proof}[Proof of Theorem \ref{tw21}] The proof of i) is elementary, so we skip the details. To show ii), let us assume first that $y_t$ is a zero-mean stationary time series. To derive (\ref{eq:eq}) notice now that to ensure $E(y_t)=0$ (see \ref{eq:mean}), we easily obtain that $\boldsymbol{\mu}=\boldsymbol{0}$. To show the remaining equalities from (\ref{eq:eq}) note that the first derivative (at $t$) of the autocovariance function of $\{y_t\}$ has the following form (under $\boldsymbol{\mu}=\boldsymbol{0}$):
\begin{equation}
\begin{split}
\frac{d E(y_t y_{t+\tau})}{d t} & = \lambda \left(\omega _{22}(\tau)-\omega _{11}(\tau)\right) \sin (\lambda (2 t+\tau ))\\
& + \lambda \left(\omega _{12}(\tau)+\omega _{21}(\tau)\right) \cos (\lambda (2 t+\tau )),
\end{split}
\end{equation}
and should be equal to zero to ensure the stationarity of $\{y_t\}$. This condition is met only under $\omega_{12}(\tau)=-\omega_{21}(\tau) \,\, \wedge \,\, \omega_{22}(\tau)=\omega_{11}(\tau)$, which concludes this part of the proof.

Conversely, if 
$\boldsymbol{\mu}=\boldsymbol{0} \,\, \wedge \,\, \omega_{12}(\tau)=-\omega_{21}(\tau) \,\, \wedge \,\, \omega_{22}(\tau)=\omega_{11}(\tau)$
then one easily gets that $E(y_t)=0$ with the autocovariance function  
$E(y_t y_{t+\tau})=\omega_{11}(\tau) \cos \lambda \tau  + \omega_{12}(\tau) \sin \lambda \tau,$
which does not depend on $t$. This concludes the proof. 
\end{proof}

\begin{proof}[Proof of Theorem \ref{tw_Gaussian_independent}]
To show that $\alpha_t$ and $\beta_t$ are both zero-mean Gaussian, it is sufficient to show that for any $k \in \mathbb{N}$, we have:
\begin{equation}
\begin{aligned}
E(\alpha_t^{2k}) &= E(\beta_t^{2k}) = \frac{2^k \sigma^{2k}\Gamma(\frac{1}{2}+k)}{\sqrt{\pi}}, \
E(\alpha_t^{2k-1}) &= E(\beta_t^{2k-1}) = 0.
\end{aligned}
\label{formula_moment}
\end{equation}
We prove this by mathematical induction.\\
For $k=1$, the above moment formula is obvious to ensure stationarity up to order 2 (see Theorem \ref{tw21}). Assume now that formula (\ref{formula_moment}) holds true up to some $k>1$. Then, under stationarity up to order $2k+1$ for ${y_t}$, the moment $E(y_t^{2 k+1})$ cannot depend on $t$. For this moment, we simply obtain:
\begin{equation}
\begin{split}
E(y_t^{2 k+1}) & = \sum\limits_{s=0}^{2k+1} E(\alpha_{t}^{s}) E(\beta_{t}^{2k+1-s}) \cos^{s} (\lambda t) \sin^{2k+1-s} (\lambda t) \\
& = E(\alpha_{t}^{2k+1}) \cos^{2k+1} (\lambda t) + E(\beta_{t}^{2k+1}) \sin^{2k+1} (\lambda t),
\end{split}
\end{equation}
which does not depend on $t$ if and only if $E(\alpha_{t}^{2k+1})=E(\beta_{t}^{2k+1})=0$. Under stationarity up to order $2(k+1)$ for ${y_t}$, we consider:
\begin{equation}
\begin{split}
E(y_t^{2 (k+1)}) & = \sum\limits_{s=0}^{2(k+1)} \binom{2(k+1)}{s}E(\alpha_{t}^{s}) E(\beta_{t}^{2(k+1)-s}) \cos^{s} (\lambda t) \sin^{2(k+1)-s} (\lambda t) \\
& = \sum\limits_{r=0}^{k+1} \binom{2(k+1)}{2r} E(\alpha_{t}^{2 r}) E(\beta_{t}^{2(k+1)-2r}) \cos^{2r} (\lambda t) \sin^{2(k+1)-2r} (\lambda t)\\
& = E(\alpha_{t}^{2(k+1)}) \cos^{2(k+1)} (\lambda t) + E(\beta_{t}^{2(k+1)}) \sin^{2(k+1)} (\lambda t)\\
& + \sum\limits_{r=1}^{k} \binom{2(k+1)}{2r} \frac{2^{k+1} \sigma^{2(k+1)}\Gamma(\frac{1}{2}+r)\Gamma(\frac{3}{2}+k-r)}{\pi} \cos^{2r} (\lambda t) \sin^{2(k+1-r)} (\lambda t)\\
& = E(\alpha_{t}^{2(k+1)}) \cos^{2(k+1)} (\lambda t) + E(\beta_{t}^{2(k+1)}) \sin^{2(k+1)} (\lambda t)\\
& -\frac{2^{k+1} \sigma^{2(k+1)} \Gamma \left(k+\frac{3}{2}\right) \left(\sin ^{2 k+2}(t)+\cos ^{2 k+2}(t)-1\right)}{\sqrt{\pi }}\\
& = \left( E(\alpha_{t}^{2(k+1)}) - \frac{2^{k+1} \sigma^{2(k+1)}\Gamma(\frac{1}{2}+k+1)}{\sqrt{\pi}} \right) \cos^{2(k+1)} (\lambda t) \\
& + \left( E(\beta_{t}^{2(k+1)}) - \frac{2^{k+1} \sigma^{2(k+1)}\Gamma(\frac{1}{2}+k+1)}{\sqrt{\pi}} \right) \sin^{2(k+1)} (\lambda t)\\
& + \frac{2^{k+1} \sigma^{2(k+1)}\Gamma(\frac{1}{2}+k+1)}{\sqrt{\pi}},
\end{split}
\nonumber
\end{equation}
which does not depend on $t$ if and only if $E(\alpha_{t}^{2(k+1)})=E(\beta_{t}^{2(k+1)})=\frac{2^{k+1} \sigma^{2(k+1)}\Gamma(\frac{1}{2}+k+1)}{\sqrt{\pi}}$. This completes the proof of formula (\ref{formula_moment}) and, therefore, the entire proof.
\end{proof}

\begin{proof}[Proof of Theorem \ref{strictly_stationary}]
Let $m$ be any positive integer and let $\boldsymbol{\tau}=(\tau_1,\tau_2,\ldots,\tau_{m}) \in \mathbb{Z}^{m}$ such that $\tau_1 < \tau_2 < \ldots < \tau_{m}$. For the vector $\textbf{y}_{t,\boldsymbol{\tau}}=[y_{t+\tau_1}\,\, y_{t+\tau_1}^{*}\,\,y_{t+\tau_2}\,\, y_{t+\tau_2}^{*} \ldots y_{t+\tau_{m}}\,\, y_{t+\tau_{m}}^{*} ]$, where $y_{t}^{*}=\alpha_t \cos \lambda t - \beta_t \sin \lambda t$, we have
\begin{equation}
\begin{split}
\textbf{y}_{t,\boldsymbol{\tau}}'& =\text{diag}(\textbf{R}(\lambda \tau_1),\ldots,\textbf{R}(\lambda \tau_{m})) \cdot \text{diag}(\textbf{R}(\lambda t),\ldots,\textbf{R}(\lambda t)) \cdot \textbf{S}_{t,\boldsymbol{\tau}}'\\
& = \text{diag}(\textbf{R}(\lambda \tau_1),\ldots,\textbf{R}(\lambda \tau_{m})) \cdot \textbf{A} \cdot \textbf{S}_{t,\boldsymbol{\tau}}',
\end{split}
\end{equation}
where $\textbf{A}=\text{diag}(\textbf{R}(\lambda t),\ldots,\textbf{R}(\lambda t))$. To finish the proof, it is enough to show that the distribution of $\textbf{A} \textbf{S}_{t,\boldsymbol{\tau}}'$ does not depend on $t$.
Since the matrix $\textbf{A}$ is orthonormal, the Jacobian of the transformation $[z_1\,\, z_1^{*}\,\,z_2\,\,z_2^{*}\,\,\ldots\,\,z_m\,\,z_m^{*}]'=\textbf{A} [x_1\,\, x_1^{*}\,\,x_2\,\,x_2^{*}\ldots\,\,x_m\,\,x_m^{*}]'$ is one. Elementary calculations give that $z_j^2 + z_j^{*2} = x_{j}^2+x_{j}^{*2}$ for $j=1,2,\ldots,m$. Hence the probability distribution function of $\textbf{A} \textbf{S}_{t,\boldsymbol{\tau}}'$ at point $(z_1, z_1^{*},z_2,z_2^{*},\ldots,z_m,z_m^{*}) \in \mathbb{R}^{2m}$ is
$f_{\boldsymbol{\tau}}(z_1^2+z_1^{*2},z_2^2+z_2^{*2},\ldots,z_m^2+z_m^{*2})$, which concludes the proof.
\end{proof}

\textcolor{black}{\begin{proof}[Proof of Theorem \ref{tw_representation_spherical}] To show i), it is enough to show that the characteristic function of $[\alpha_t\,\,\beta_t]$ at $\textbf{t}=[t_1,\,\, t_2] \in \mathbb{R}^2$ is a function of $\| \textbf{t} \| = \sqrt{t_1^2 + t_2^2}$. By $\psi_{\epsilon}(\|\textbf{t}\|)$ we denote the characteristic function of $[\epsilon_t\,\,\epsilon_t^*]$ at point $\textbf{t} \in \mathbb{R}^2$. Notice that the characteristic function of $[\alpha_t\,\,\beta_t]$ at $\textbf{t}$ is 
\begin{equation}
\phi(\textbf{t})=E(e^{i \textbf{t} [\alpha_t\,\,\beta_t]'})=\prod\limits_{k=0}^{\infty} E(e^{i \psi_k \textbf{t} [\epsilon_t\,\,\epsilon_t^*]'}) =  \prod\limits_{k=0}^{\infty} \phi_{\epsilon}(|\psi_k| \cdot \|\textbf{t}\|),
\end{equation}
which ends the proof of i).\\
To show ii) note that using elementary algebra it can be shown that since the rotation matrix is orthogonal (with its determinant equal one), the probability distribution function of $[\zeta_t\,\,\zeta_t^*]$ at point $(x,y) \in \mathbb{R}^2$ is $f(x^2+y^2)$ (we omit technical details). Now define $y_{t}^*=\sum\limits_{k=0}^{\infty} \psi_{k} [ \cos (\lambda k) \zeta_{t-k} -\sin (\lambda k) \zeta_{t-k}^*] $ and notice that 
\begin{equation}
\begin{split}
[y_t\,\, y_t^*]'& = \sum\limits_{k=0}^{\infty} \psi_{k} \textbf{R}(\lambda k) [\zeta_{t-k}\,\,\zeta_{t-k}^*]'=\sum\limits_{k=0}^{\infty} \psi_{k} \textbf{R}(\lambda k) \textbf{R}(\lambda (t-k))[\epsilon_{t-k}\,\,\epsilon_{t-k}^*]'\\
& = \textbf{R}(\lambda t) \sum\limits_{k=0}^{\infty} \psi_{k}  [\epsilon_{t-k}\,\,\epsilon_{t-k}^*]'= \textbf{R}(\lambda t)   [\alpha_{t}\,\,\beta_{t}^*]',
\end{split}
\end{equation}
which ends the proof of ii).\\
To show iii), it is sufficient to use the representation from ii). We skip technical details. \\
This ends the proof.
\end{proof}}

\begin{proof}[Proof of Theorem \ref{tw:22}] One easily checks that (\ref{y_rep}) holds if (\ref{alpha_rep}) is substituted into into the former. To prove the stationarity, note that for any $t, \tau \in \mathbb{Z}$, we have
\begin{equation}
\begin{split}
E([\alpha_t \,\,\, \beta_t]'[\alpha_{t+\tau} \,\,\, \beta_{t+\tau}]) & \!=\! \textbf{R} (\lambda t) E([y_t \,\,\, y_t^*]'[y_{t+\tau} \,\,\, y_{t+\tau}^*]) \textbf{R} (-\lambda t) \textbf{R} (-\lambda \tau ) \\
& \!=\!\boldsymbol{\Gamma}(\tau) \textbf{R} (-\lambda \tau),
\end{split}
\end{equation}
which concludes this part of the proof. Proving (\ref{eq:omega2}) is elementary, therefore we omit this part of the proof. This completes the proof. 
\end{proof}

\begin{proof}[Proof of Theorem \ref{tw:psd}] Note that since $y_t$ is zero-mean and $[\alpha_{t,j}\,\,\,\beta_{t,j}]$ are mutually uncorrelated, one gets
\begin{equation}
\begin{split}
& \gamma_{y}(\tau)=E(y_t y_{t+\tau})=\\
& \sum\limits_{j=1}^{K}  (\alpha_{t+\tau,j} \cos[ \lambda (t+\tau) ] + \beta_{t+\tau,j} \sin[ \lambda (t+\tau) ])(\alpha_{t,j} \cos \lambda t + \beta_{t,j} \sin \lambda t)= \\
& \sum\limits_{j=1}^{K} \gamma_j (\tau) \cos \lambda_j \tau.
\end{split}
\nonumber
\end{equation}
Hence, the power spectral density of $y_t$ is given by
\begin{equation}
\begin{split}
f_{y}(\omega) & =\sum\limits_{\tau=-\infty}^{\infty} \gamma_{y}(\tau) \cos \omega \tau = \sum\limits_{j=1}^{K} \sum\limits_{\tau = \infty}^{\infty} \gamma_{j}(\tau) \cos \lambda_j \tau \cos \omega \tau \\
& = \frac{1}{2} \sum\limits_{j=1}^{K} \big(f_{j}(\lambda_j+\omega)  + f_{j}(\lambda_j - \omega)\big),
\end{split}
\nonumber
\end{equation}
which ends the proof. 
\end{proof}
\begin{proof}[Proof of Theorem \ref{tw:representation}] Let $[\epsilon_t\,\, \epsilon_{t}^*]$ have spherical distribution with characteristic function at $z \in \mathbb{R}^2$ given by $\phi(\| z\|)$.  Then we have decomposition $[\epsilon_t\,\, \epsilon_{t}^*]=\sum\limits_{j=1}^K[\epsilon_{j,t}\,\, \epsilon_{j,t}^*]$, where $\{[\epsilon_{j,t}\,\, \epsilon_{j,t}^*]\}$ ($j=1,2,\ldots,K$) is sequence of zero-mean IID random vectors, mutually independent and having spherical distribution given by characteristic function $\phi^{k_j}(\|z\|)$, with $k_j>0$  and $\sum\limits_{j=1}^{K}k_j=K$. Then, we define processes $[y_{j,t}\,\,y_{j,t}^*]=\Psi(L)[\epsilon_{j,t}\,\, \epsilon_{j,t}^*]$, which are also mutually independent for $j=1,2,\ldots,K$. Next, for any $j=1,2,\ldots,K$, we define:
\begin{equation}
[\alpha_{j,t}\,\,\beta_{j,t}]'=\textbf{R}(-\lambda_j t)[y_{j,t}\,\,y_{j,t}^*]'.
\end{equation}
Elementary calculations yield (\ref{eq:yt}). This ends the proof.
\end{proof}

\begin{proof}[Proof of Theorem \ref{Hannan_to_Harvey}] To show i) notice that for any $j=1,2,\ldots,K$, a single component $z_{j,t}$ satisfies the bivariate model equation: 
\begin{equation}
\begin{split}
\left[
\begin{array}{c}
 z_{j,t} \\
 z_{j,t}^{*} \\
\end{array}
\right]&= \textbf{R}(t \lambda_j)   \left[
\begin{array}{c}
 \alpha_{j,t} \\
 \beta_{j,t} \\
\end{array}
\right]\\
& =\textbf{R}(t \lambda_j)  \left( \sum\limits_{i=1}^{p_j}\phi_{j,i} \left[
\begin{array}{c}
 \alpha_{j,t-i} \\
 \beta_{j,t-i} \\
\end{array}
\right] + \left[
\begin{array}{c}
 \epsilon_{j,t} \\
 \epsilon_{j,t}^* \\
\end{array}
\right]\right)\\
& =\textbf{R}(t \lambda_j)  \left( \sum\limits_{i=1}^{p_j}\phi_{j,i} \textbf{R}(-t \lambda_j+i\lambda_j)   \left[
\begin{array}{c}
 z_{j,t-i} \\
 z_{j,t-i}^* \\
\end{array}
\right] + \left[
\begin{array}{c}
 \epsilon_{j,t} \\
 \epsilon_{j,t}^* \\
\end{array}
\right]\right)\\
& =  \sum\limits_{i=1}^{p_j}\phi_{j,i} \textbf{R}(i\lambda_j)   \left[
\begin{array}{c}
 z_{j,t-i} \\
 z_{j,t-i}^* \\
\end{array}
\right] + \textbf{R}(t \lambda_j) \left[
\begin{array}{c}
 \epsilon_{j,t} \\
 \epsilon_{j,t}^* \\
\end{array}
\right],
\end{split}
\nonumber
\end{equation}
where $\textbf{R}(t \lambda_j) \left[
\begin{array}{c}
 \epsilon_{j,t} \\
 \epsilon_{j,t}^* \\
\end{array}
\right]$ is a bivariate white noise with the same distribution as $[\epsilon_{j,t} \,\, \epsilon_{j,t}^{*}]'$. \\
\textcolor{black}{To show ii) notice that from i) the vector $[z_{1,t}\,\,z_{1,t}^*\,\,z_{2,t}\,\,z_{2,t}^*\,\, \ldots \,\,z_{K,t}\,\,z_{K,t}^*]$ admits a VAR representation with the white noise of the form $\text{diag}(\boldsymbol{R}(\lambda_1 t),\boldsymbol{R}(\lambda_2 t),\ldots,\boldsymbol{R}(\lambda_K t))\boldsymbol{\epsilon}_t'$ and having the same pdf as $\boldsymbol{\epsilon}_t'$, which is a consequence of the orthogonality of the matrix $\text{diag}(\boldsymbol{R}(\lambda_1 t),\boldsymbol{R}(\lambda_2 t),\ldots,\boldsymbol{R}(\lambda_K t))$ (with its determinant equal one).} 
\end{proof}


We continue  with the following lemmas. 
\begin{lemma}
Let $X \sim N(\mu,\sigma^2)$. Then $E(\sin X)=e^{-\frac{\sigma ^2}{2}} \sin (\mu )$, and $E(\cos X)=e^{-\frac{\sigma ^2}{2}} \cos (\mu )$.
\label{lemat1}
\end{lemma}
\begin{proof}
We omit an elementary proof.    
\end{proof}
\begin{lemma}
Let $P_t$ follows an $\text{ARIMA}(p,1,q)$ model of the form: 
\begin{equation}
(1-L)\Psi(L)P_t=\Theta(L) \epsilon_t,    
\end{equation}
where $\epsilon_t$ is a white noise. Let $s$ be any positive integer, $\{\tau_1,\tau_2,\ldots,\tau_s\}$ be any sequence of integers, and $\{b_1,b_2,\ldots,b_s\}$ be any sequence of real numbers. Then, for the process $Q_{t+\tilde \tau}=\sum\limits_{k=1}^{s} b_k P_{t+\tau_k}$, we have the following ARIMA representation:
\begin{equation}
(1-L)\Psi(L)Q_t=\Theta(L) \left(\sum\limits_{k=1}^{s} b_k L^{\tilde \tau - \tau_k}\right) \epsilon_t,    
\end{equation}
where $\tilde \tau =\max\limits_{i=1,2,\ldots,s}\{\tau_i\}$.
\label{lematQ}
\end{lemma}
\begin{proof}
Note that $P_t=P_{t-1} + \frac{\Theta(L)}{\Psi(L)} \epsilon_t$. Hence, 
\begin{equation}
\begin{split}
Q_{t+\tilde \tau}& =\sum\limits_{k=1}^{s} b_k P_{t+\tau_k} = \sum\limits_{k=1}^{s} b_k P_{t+\tau_k-1} +  \sum\limits_{k=1}^{s} b_k \frac{\Theta(L)}{\Psi(L)} \epsilon_{t+\tau_k}\\
& = Q_{t+\tilde \tau-1} +  \sum\limits_{k=1}^{s} b_k \frac{\Theta(L)}{\Psi(L)} \epsilon_{t+\tau_k}.
\end{split}
\end{equation}
Therefore, 
\begin{equation}
(1-L)Q_{t+\tilde \tau}=  \sum\limits_{k=1}^{s} b_k \frac{\Theta(L)}{\Psi(L)} \epsilon_{t+\tau_k},  
\end{equation}
and finally,
\begin{equation}
(1-L)\Psi(L)Q_{t+\tilde \tau}=   \Theta(L) \left( \sum\limits_{k=1}^{s} b_k L^{\tilde \tau - \tau_k} \right) \epsilon_{t+\tilde \tau},  
\end{equation}
which ends the proof upon noticing that $Q_{t+\tilde \tau}=L^{-\tilde \tau} Q_{t}$ and $\epsilon_{t+\tilde \tau}=L^{-\tilde \tau} \epsilon_{t}$
\end{proof}
\begin{lemma}
Let $s(t)$ be any real-valued function, and $W_t$ be any $\text{ARIMA}(p,1,q)$ process with a Gaussian error term $\epsilon_t$: $(1-L)\Psi(L) W_t=\Theta(L) \epsilon_t.$ Then, $$E(\sin (s(t)+W_t))=E(\cos (s(t)+W_t))=0.$$ 
\label{lematX}
\end{lemma}
\begin{proof}
Let $W_t=W_{t-1} + \tilde \epsilon_t$, where $\tilde \epsilon_t = \frac{\Theta(L)}{\Psi(L)} \epsilon_t = \sum\limits_{j=0}^{\infty} \phi_{j}\epsilon_{t-j}$ is a zero-mean Gaussian stationary time series.  Then, for any positive integer $d$, we have $W_t=W_{t-d}+\epsilon_{d,t}^*$, where 
\begin{equation}
\begin{split}
\epsilon_{d,t}^*&=\sum\limits_{s=0}^{d-1}\tilde \epsilon_{t-s}=\sum\limits_{s=0}^{d-1} \sum\limits_{j=0}^{\infty} \phi_{j}\epsilon_{t-s-j}= \sum\limits_{j=0}^{\infty} \phi_{j} \sum\limits_{s=0}^{d-1} \epsilon_{t-(s+j)}\\
&=\underbrace{\sum\limits_{k=0}^{d-1} \left(\sum\limits_{j=0}^k\phi_{k}\right) \epsilon_{t-k}}_{\zeta_{d,t}^{(1)}} + \underbrace{\sum\limits_{k=d}^{\infty} \left(\sum\limits_{j=0}^{d-1}\phi_{k}\right) \epsilon_{t-k}}_{\zeta_{d,t}^{(2)}}    
\end{split}    
\end{equation} is a zero-mean Gaussian random variable. Note that 
\begin{equation}
\epsilon_{d,t}^*| \mathcal{F}_{t-d} \sim N(\zeta_{d,t}^{(2)},\sigma_{d}^2),
\label{normal1}
\end{equation} 
where $\sigma_{d}^2=Var(\zeta_{d,t}^{(1)})$ does not depend on $t$ and (skipping technical details)  
\begin{equation}
 \lim_{d \to \infty} \sigma_{d}^2 = \infty.
 \label{varinfty}
\end{equation}
Hence, by a simple trigonometric identity, and employing (\ref{normal1}) with Lemma \ref{lemat1}, one gets: 
\begin{equation}
\begin{split}
 & |E(\sin (s(t)+W_t))| = \left|E\left(\sin[s(t)+  W_{t-d} + \epsilon_{d,t}^*] \right) \right|\\
 & = \left|E\left(\sin[s(t)+ W_{t-d}]\cos( \epsilon_{d,t}^*) + \cos[s(t)+ W_{t-d}]\sin(\epsilon_{d,t}^*) \right)\right|\\
  & = \Bigg|E\bigg(\sin[s(t) + W_{t-d}]E\big(\cos( \epsilon_{d,t}^*)\big|t-d\big)\\
  & \,\,\,\,\,\, + \cos[s(t) + W_{t-d}]E\big(\sin(  \epsilon_{d,t}^*)\big|t-d\big) \bigg)\Bigg|\\
  & = e^{-\frac{\sigma^2_{d}}{2}} \bigg|E\bigg(\sin[s(t) + W_{t-d}]\cos( \zeta_{d,t}^{(2)})\\
  & \,\,\,\,\,\, + \cos[s(t) + W_{t-d})]\sin( \zeta_{d,t}^{(2)}) \bigg)\bigg|\\
  & \leq 2 e^{-\frac{\sigma^2_{d}}{2}}.
 \end{split}
\end{equation}
Taking the limit $d \to \infty $ of both sides of the above inequality, and employing (\ref{varinfty}), we obtain $E(\sin (s(t)+W_t))=0$. The proof for $E(\cos (s(t)+W_t))$ is analogous. This ends the proof.
\end{proof}
\begin{proof}[Proof of Theorem \ref{tw1}]\textcolor{black}{
This result is a natural consequence of the product-to-sum identity and Lemma \ref{lemat1}. Therefore, only major steps are shown. Since $A_t$ and $P_t$ are independent, we have 
\begin{equation}
\begin{split}
& E(y_{y+\tau_1}y_{t+\tau_2}\cdot\ldots\cdot y_{t+\tau_s}) = E\left(\prod\limits_{k=1}^{s}(a+A_{t+\tau_k})\sin[\lambda(t+\tau_k+P_{t+\tau_k})]\right)\\
& = \underbrace{E\left(\prod\limits_{k=1}^{s}(a+A_{t+\tau_k})\right)}_{I} \underbrace{E \left(\prod\limits_{k=1}^{s}\sin[\lambda(t+\tau_k+P_{t+\tau_k})]\right)}_{II}.
\end{split}    
\end{equation}
Applying now the product-to-sum identity and Lemma \ref{lemat1} to term II, we get
\begin{equation}
\begin{split}
II & = E \left(\prod\limits_{k=1}^{s}\sin[\lambda(t+\tau_k+P_{t+\tau_k})]\right)=\\
&=\frac{(-1)^{\lfloor s/2 \rfloor}}{2^s} \sum\limits_{\boldsymbol{e} \in P} E\left(f\left(\sum\limits_{j=1}^s e_j \lambda(t+\tau_j)+\sum\limits_{j=1}^s e_jP_{t+\tau_j}  \right)\right)\\
& = \frac{(-1)^{\lfloor s/2 \rfloor}}{2^s} \sum\limits_{\boldsymbol{e} \in P}  e^{-\frac{1}{2}\text{Var}\left(\sum\limits_{j=1}^s e_jP_{t+\tau_j}\right)}   f\left(\lambda \sum\limits_{j=1}^s e_j (t+\tau_j)\right),
\end{split}    
\end{equation}
where $\boldsymbol{e}=(e_1,e_2,\ldots,e_s) \in \{-1,1\}^s$ and $f$ is a sine function if $s$ is odd, and a cosine function otherwise. This ends the proof.}
\end{proof}
\begin{proof}[Proof of Theorem \ref{tw2}]
We split the proof into two steps. First, we show that $E(y_t)=0$, and then, we prove that for any integer $s \geq 2$ and any integers $\tau_1 \leq \tau_2 \leq \ldots \leq \tau_s$,  the expectation $E(y_{y+\tau_1}y_{t+\tau_2}\cdot\ldots\cdot y_{t+\tau_s})$ exists and does not depend on $t$.\\\\
\textit{Step 1.} Note that $E(y_t) =a E\left(\sin[\lambda t+ \lambda P_t)] \right)$. Hence, using Lemma (\ref{lematX}), we obtain $E(y_t)=0$. \\\\
\textit{Step 2.} We have 
\begin{equation}
\begin{split}
& E(y_{y+\tau_1}y_{t+\tau_2}\cdot\ldots\cdot y_{t+\tau_s}) = E\left(\prod\limits_{k=1}^{s}(a+A_{t+\tau_k})\sin[\lambda(t+\tau_k+P_{t+\tau_k})]\right)\\
& = \underbrace{E\left(\prod\limits_{k=1}^{s}(a+A_{t+\tau_k})\right)}_{I} \underbrace{E \left(\prod\limits_{k=1}^{s}\sin[\lambda(t+\tau_k+P_{t+\tau_k})]\right)}_{II}.
\end{split}    
\end{equation}
Since $A_t$ is stationary up to order $m$, the term $I$ does not depend on $t$. To finish this step it is enough to show that the term $II$ also does not depend on $t$. To that end, notice that using the product-to-sum identity we get
\begin{equation}
\begin{split}
II & = E \left(\prod\limits_{k=1}^{s}\sin[\lambda(t+\tau_k+P_{t+\tau_k})]\right)=\\&=\frac{(-1)^{\lfloor s/2 \rfloor}}{2^s} \sum\limits_{\boldsymbol{e} \in P} \underbrace{E\left(f\left(\overbrace{\sum\limits_{j=1}^s e_j \lambda(t+\tau_j)}^{III_{\boldsymbol{e}}}+\overbrace{\sum\limits_{j=1}^s e_jP_{t+\tau_j}}^{S_{\boldsymbol{\tau}}}  \right)\right)}_{II_{\boldsymbol{e}}},
\end{split}    
\end{equation}
where $\boldsymbol{e}=(e_1,e_2,\ldots,e_s) \in \{-1,1\}^s$ and $f$ is a sine function if $s$ is odd, and a cosine function otherwise. Thus, to finish the proof of Step 2, it suffices to show that $II_{\boldsymbol{e}}$ does not depend on $t$ for any $\boldsymbol{e} \in P$. To that end, two cases are considered:
\begin{itemize}
\item[Case 1.] If $\sum\limits_{j=1}^s e_j=0$, then $s$ is even and $III_{\boldsymbol{e}}$ does not depend on $t$, and by Lemma \ref{lematQ}, the time series $S_{\boldsymbol{\tau}}$ is a zero-mean Gaussian $\text{ARMA}$ process. Therefore, the exact form of $II_{\boldsymbol{e}}$ follows from Lemma \ref{lemat1}, and it does not depends on $t$. The function $f$ is a cosine function here, and hence, by this Lemma, we get
\begin{equation}
II_{\boldsymbol{e}}=E(\cos(III_{\boldsymbol{e}}+S_{\boldsymbol{e}})) = \sqrt{e^{-\text{var}(S_{\boldsymbol{e}})} }\cos(III_{\boldsymbol{e}}),  
\end{equation}
which finishes the proof of Case 1. 
\item[Case 2.] If $\sum\limits_{j=1}^s e_j \not = 0$, then by Lemma \ref{lematQ}, the time series $S_{\boldsymbol{\tau}}$ is an $\text{ARIMA}(p_{\boldsymbol{\tau}},1,q_{\boldsymbol{\tau}})$ process. Thus, resorting to Lemma \ref{lematX} results in $II_{\boldsymbol{e}}=0$. 
\end{itemize}
\vspace{0.5 cm}This ends the proof of Step 2 and completes the entire proof. 
\end{proof}

\bibliographystyle{apalike}
\bibliography{bib.bib}

\end{document}